\documentclass[11pt]{amsart}
\usepackage{a4wide}
\usepackage[all,cmtip]{xy}
\usepackage{amssymb}
\usepackage{mathabx}

\usepackage{amscd,amsmath,amssymb,fancyhdr,color}
\usepackage{tikz-cd}
\usepackage{scalerel, mathtools}
\usepackage[backref=page]{hyperref}
\renewcommand*{\backref}[1]{}

\newcommand{\la}{\lambda}
\newcommand{\al}{\alpha}
\newcommand{\be}{\beta}

\newcommand{\ov}{\overline}

\newcommand{\Ker}{\text{Ker}}

\newcommand{\pa}[1]{\ensuremath{\frac{\partial}{\partial #1}}}

\newcommand{\Ka}{K\"ahler}

\newcommand{\CC}{\mathbb{C}}
\newcommand{\HH}{\mathbb{H}}
\newcommand{\RR}{\mathbb{R}}
\newcommand{\ZZ}{\mathbb{Z}}
\newcommand{\NN}{\mathbb{N}}

\newcommand{\Ss}{\mathbb{S}}

\newcommand{\cz}{\widebar{z}}

\newcommand{\e}{\mathrm{e}}

\newcommand{\del}{\partial}

\newcommand{\LL}{\mathcal{L}}
\newcommand{\TT}{\mathbb{T}}
\newcommand{\ce}{\mathcal{C}^\infty}

\numberwithin{equation}{section}

\def\eqref#1{(\ref{#1})}

\newcommand{\C}{{\mathbb C}}
\newcommand{\R}{{\mathbb R}}
\newcommand{\Q}{{\mathbb Q}}
\renewcommand{\H}{{\mathbb H}}

\def\1{\sqrt{-1}\:}

\newcommand{\cntrct}                
{\hspace{2pt}\raisebox{1pt}{\text{$\lrcorner$}}\hspace{2pt}}

\renewcommand{\phi}{\varphi}
\renewcommand{\epsilon}{\varepsilon}
\renewcommand{\geq}{\geqslant}
\renewcommand{\leq}{\leqslant}

\newcommand{\Hom}{\operatorname{Hom}}
\newcommand{\Aut}{\operatorname{Aut}}

\newcommand{\Diff}{\operatorname{Diff}}

\renewcommand{\dim}{\operatorname{dim}}

\renewcommand{\Re}{\operatorname{Re}}
\renewcommand{\Im}{\operatorname{Im}}

\newcounter{Mycounter}[section]
\newcounter{lemma}[section]
\newcounter{claim}[section]

\newcounter{sublemma}[section]

\newcounter{corollary}[section]

\newcounter{theorem}[section]

\newcounter{conjecture}[section]

\newcounter{proposition}[section]

\newcounter{definition}[section]

\newcounter{example}[section]

\newcounter{remark}[section]

\newcounter{problem}[section]

\newcounter{question}[section]

\makeatletter

\@addtoreset{equation}{section}

\@addtoreset{footnote}{section}

\makeatother

\title{De Rham and Twisted Cohomology of Oeljeklaus-Toma manifolds}
\author{Nicolina Istrati}
\address[Nicolina Istrati]{Univ Paris Diderot, Sorbonne Paris Cit\'{e}
Institut de Math\'{e}matiques de Jussieu-Paris Rive Gauche 
Case 7012, 75205 Paris Cedex 13, France}
\email{nicolina.istrati@imj-prg.fr}
\author{Alexandra Otiman} 
\thanks{A. O. is partially supported by a grant of Ministry of Research and Innovation, CNCS - UEFISCDI,
project number PN-III-P4-ID-PCE-2016-0065, within PNCDI III}

\address[Alexandra Otiman]{
Institute of Mathematics "Simion Stoilow" of the Romanian Academy, 21,
Calea Grivitei Street, 010702, Bucharest, Romania}
\address{Max Planck Institut f\"ur Mathematik, Vivatsgasse 7, 53111 Bonn, Germany}
\address{University of Bucharest, Research Center in Geometry, Topology and Algebra, Faculty of Mathematics and Computer Science, 14 Academiei Str., Bucharest, Romania}
\email{aotiman@mpim-bonn.mpg.de, alexandra\_otiman@yahoo.com}



\begin{document}
\begin{abstract} Oeljeklaus-Toma (OT) manifolds are complex non-K\"ahler manifolds whose construction arises from specific number fields. In this note, we compute their de Rham cohomology in terms of invariants associated to the background number field. This is done by two distinct approaches, one by averaging over a certain compact group, and the other one using the Leray-Serre spectral sequence. In addition, we compute also their  twisted  cohomology. As an application, we show that the low degree Chern classes of any complex vector bundle on an OT manifold vanish in the real cohomology. Other applications concern the OT manifolds admitting locally conformally \Ka\ (LCK) metrics: we show that there is only one possible Lee class of an LCK metric, and we determine all the possible  twisted  classes of an LCK metric, which implies the nondegeneracy of certain Lefschetz maps in cohomology.  
\end{abstract}
\maketitle
\section{Introduction}

Oeljeklaus-Toma manifolds, introduced by K. Oeljeklaus and M. Toma in \cite{ot}, are compact complex non-\Ka\ manifolds which are higher dimension analogues of Inoue surfaces of type $\mathcal{S}^0$ (see \cite{i}). They are quotients of $\H^s\times\C^t$ by discrete groups of affine transformations arising from a number field $K$ and a particular choice of a subgroup of units $U$ of $K$. They are commonly referred to as OT manifolds of type $(s, t)$, and denoted by $X(K,U)$. These manifolds have been of particular interest for locally conformally K\"ahler (LCK) geometry. When they were introduced, OT manifolds of type $(s, 1)$ were shown to carry LCK metrics and they constituted the first examples of manifolds to disprove a conjecture of Vaisman, according to which the odd index Betti numbers of an LCK manifold should be odd. As a matter of fact, in higher dimension, all other explicit examples of LCK manifolds that are known in the literature admit an \textit{LCK metric with potential} (cf. \cite{ov10}), which makes them share similar properties with the Hopf manifolds. For this reason, understanding OT manifolds, which do not admit such metrics, should give more insight into LCK geometry. 

So far, significant advances have been made in the study of OT manifolds. Many of their properties are closely related to the arithmetical properties of $(K,U)$, as can be seen particularly in the papers of M. Parton and V. Vuletescu \cite{pv} and of O. Braunling \cite{b}. OT manifolds were shown to carry the structure of a solvmanifold by H. Kasuya \cite{k}, and those of type $(s,1)$ to contain no non-trivial complex submanifolds by L. Ornea and M. Verbitsky \cite{ov11}. A delicate issue seems to be the existence of LCK metrics on OT manifolds which are not of type $(s,1)$. Some progress in this direction has been made by V. Vuletescu \cite{vv} and A. Dubickas \cite{d}, but the question remains open in general. 

Concerning the cohomology of OT manifolds, their first Betti number and the second one for a certain subclass of manifolds, called \textit{of simple type}, were computed in \cite{ot}. More recently, H. Kasuya computed in \cite{k2} the de Rham cohomology of OT manifolds of type $(s,1)$, using their solvmanifold structure. In this note, we use different methods to compute the de Rham cohomology algebra as well as the  twisted  cohomology of any OT manifold $X(K,U)$. This is done in terms of numerical invariants coming from $U\subset K$, and the exact statements are \ref{teoremaA} and \ref{teoremaB}. \ref{teoremaA} is proved by two different approaches, one by reducing to the invariant cohomology with respect to a certain compact Lie group, in Section 3, and the other one using the Leray-Serre spectral sequence, in Section 4. This last approach is also used to prove \ref{teoremaB} in Section 5.

The last section is devoted to a few applications, focusing on the OT manifolds which admit an LCK metric. We compute the explicit cohomology for OT manifolds admitting LCK metrics (\ref{LCKbun}, \ref{twistedLCK}) and for OT manifolds associated to a certain family of polynomials (\ref{exemplu}). Also, we show that the set of possible Lee classes for an LCK metric on $X(K, U)$ consists of only one element (\ref{unic}). The problem of characterizing the set of Lee classes was first considered by Tsukada in \cite{ts}, who gave a description for Lee classes of LCK metrics on Vaisman manifolds and then by Apostolov and Dloussky in \cite{ad}, where they characterize this set for Hopf surfaces, Inoue surfaces $\mathcal{S}^{\pm}$ and Kato surfaces. In \cite{alex}, it is described also for Inoue surfaces of type $\mathcal{S}^0$ and finally, OT manifolds complete the list of known LCK manifolds for which this set is known. Additionally, we determine all the possible  twisted  classes of LCK forms on OT manifolds (\ref{LCKclass}), generalizing a result of \cite{alex} showing that this class cannot vanish. They all turn out to induce a non-degenerate Lefschetz map in cohomology. A final application (\ref{Chern}) concerns the vanishing of certain real Chern classes of vector bundles on OT manifolds.

\section{Preliminaries}

\subsection{Oeljeklaus-Toma manifolds}

We briefly recall the construction of Oeljeklaus-Toma manifolds, following \cite{ot}, and some of their properties that we will need.

Given two positive numbers $s,t>0$, an OT manifold $X$ of type $(s,t)$ is a compact quotient of $\tilde X:=\HH^s\times\CC^t$ by a discrete group $\Gamma$ of rank $2(s+t)$ arising from a number field. More specifically, let $m=s+t$ and $n=s+2t$ and let $K$ be a number field with $n$ embeddings in $\C$, $s$ of them real and $2t$ complex conjugated. We shall denote  these embeddings by $\sigma_1, \ldots, \sigma_{n}$, with the convention that the first $s$ are real and $\sigma_{s+t+i}=\overline{\sigma_{s+i}}$, for any $1\leq i\leq t$. The ring of integers of $K$, $O_K$, which as a $\ZZ$-module is free of rank $n$, acts on $\HH^s\times\CC^t$ via the first $m$ embeddings. If $(w,z)=(w_1,\ldots, w_s, z_1, \ldots, z_t)$ denote the holomorphic coordinates on $\HH^s\times\CC^t$, the action is given by translations:
\begin{equation*}
T_a(w,z)= (w_1 + \sigma_1(a), \ldots, w_s+\sigma_s(a), z_1 + \sigma_{s+1}(a), \ldots z_t + \sigma_{s+t}(a)),\ \ \ a\in O_K.
\end{equation*}
It is a free and proper action, and as a smooth manifold,  the quotient is given by:
 \begin{equation*}
 \hat X:=\HH^s\times \CC^t/O_K\cong(\RR_{>0})^s\times \TT^{n}.
 \end{equation*}
 Next, one defines inside the group of units $O_K^*$ the subgroup of positive units $O_K^{*, +}$ as:
\begin{equation*}
O_K^{*, +} =\{u \in O_K^* \mid \sigma_i(u) > 0, 1\leq i\leq s\}.
\end{equation*}
This group acts on $\HH^s\times \CC^t$ by dilatations as:
\begin{equation*}
R_u(w, z)=(\sigma_1(u)w_1, \ldots, \sigma_s(u)w_s, \sigma_{s+1}(u)z_1, \ldots, \sigma_{s+t}(u)z_t), \ \ \ u\in O_K^{*,+}.
\end{equation*}
This action is free, but not properly discontinuous. However, as shown in \cite{ot}, one can choose a rank $s$ subgroup $U$ in $O_K^{*,+}$ which embeds as a lattice in $(\RR_{>0})^s$ via:
\begin{align*}
j:U&\rightarrow (\RR_{>0})^s\\
u&\mapsto (\sigma_1(u),\ldots,\sigma_s(u)).
\end{align*}
We will denote by $U_\HH\cong U$ the lattice $j(U)$. In particular, $U$ acts properly discontinuously on $\HH^s\times\CC^t$. Clearly, $U$ also acts on $O_K$, so that one gets a free, properly discontinuous action of the semi-direct product $\Gamma:=U\rtimes O_K$ on $\HH^s\times \CC^t$. The quotient of this action $X:=\HH^s\times \CC^t/\Gamma$, denoted by $X(K,U)$, is the Oeljeklaus-Toma manifold of type $(s,t)$ associated to $K$ and $U$.

Since the action of $\Gamma$ on $\tilde X$ is holomorphic, $X$ is a complex manifold. Moreover, it is compact, because it has in fact the structure of a torus fiber bundle over another torus:
\begin{equation}\label{fiberbundle}
{\mathbb{T}}^{n}\rightarrow X(K, U) \xrightarrow{\pi} {\mathbb{T}}^s.
\end{equation}
Indeed, this last assertion can be seen as follows: the natural projection 
\begin{equation*}
\hat \pi:\hat X=(\RR_{>0})^s\times \TT^n\rightarrow (\RR_{>0})^s
\end{equation*} 
is a trivial $\TT^n$-fiber bundle over $(\RR_{>0})^s$. The group $U$ acts on $\hat X$, but also on $(\RR_{>0})^s$ by translations via $j$, and $\hat \pi$ is equivariant for this action. As $X=\hat X/U$ and $\TT^s=(\RR_{>0})^s/U$, $\hat\pi$ descends to the $\TT^n$-fiber bundle \eqref{fiberbundle}.

For later use, it is important to note that $\pi$ is a flat fiber bundle, meaning that it has locally constant transition functions. This is equivalent to saying that it is given by a representation $R:\pi_1(\TT^s)=U_\HH\rightarrow \Diff(\TT^n)$, $u\mapsto \Phi_u$. We can make $R$ explicit, after identifying $U_\HH$ with $U$.   Recalling that $\TT^n=\RR^s\times\CC^t/O_K$ and denoting by $r=(r_1,\ldots, r_s)$ the real coordinates on $\RR^s$, for any $u\in U$, $\Phi_u$ is given by: 
\begin{equation}\label{reprezentare}
\Phi_u((r,z) \mod O_K)=(\sigma_1(u)r_1,\ldots, \sigma_s(u)r_s,\sigma_{s+1}(u)z_1,\ldots,\sigma_m(u)z_t))\mod O_K
\end{equation}
which is clearly well defined, since $u \cdot O_K = O_K$, for all $u \in U$.

The tangent bundle of $\hat X$ splits smoothly as  $T\hat X=E\oplus V$, where $E$ is the pullback of $T(\RR_{>0})^s$ and $V$ is the pullback of $T\TT^n$ on $\hat X$ by the natural projections. 
These bundles are trivial, and for later use we will need to fix a global frame of $V^*\otimes\CC$ over $\hat X$. If $z_1, \ldots, z_t$ denote the holomorphic coordinates on $\CC^t$ and $w_1=r_1+iv_1, \ldots, w_s=r_s+iv_s$ are holomorphic coordinates on $\HH^s$, also viewed as local coordinates on $\hat X$, we choose as a basis of $\ce(\hat X, V^*\otimes \CC)$ over $\ce(\hat X,\CC)$: 
\begin{align*}
e_j=dr_j \text{ for } 1\leq j\leq s, e_{s+j}=dz_j \text{ and } e_{s+t+j}=d\cz_j \text{ for } 1\leq j\leq t.  
\end{align*}
In particular, for any $0\leq l\leq n$, a frame for $\bigwedge^l V^*\otimes \CC$ is given by: 
\begin{equation}\label{frame}
\{e_I=e_{i_1}\wedge\ldots\wedge e_{i_l}|I=(0<i_1<\ldots<i_l\leq n)\}.
\end{equation} 
For any multi-index $I=(0<i_1<\ldots<i_l\leq n)$, let us denote by $\sigma_I:U\rightarrow \CC^*$ the representation:
\begin{equation}\label{repr}
\sigma_I(u)=\sigma_{i_1}(u)\cdots\sigma_{i_l}(u).
\end{equation}
Then an element $u\in U$ acts on $e_I$ by $u^*e_I=\sigma_I(u)e_I$.

\subsection{Leray-Serre spectral sequence of a locally trivial fibration}\label{leray}

In this section, we review the general properties of the Leray-Serre spectral sequence associated to a fiber bundle. For a thorough presentation of spectral sequences and the Leray-Serre sequence we refer to \cite{gh} and \cite{botttu}.
Let $F \rightarrow X \xrightarrow{\pi} B$ be a locally trivial fibration. For a trivializing open set $U\subset B$ for $\pi$, we denote by $\phi_U$ an isomorphism $\phi_U: \pi^{-1}(U) \rightarrow U \times F$, and for two trivialising open sets $U,V$, we denote by $g_{UV} = \phi_U \circ \phi_V^{-1}$ the corresponding transition function. Let us also denote by $\mathcal{X}^v$ the sheaf of vertical vector fields on $X$, i.e. the vector fields tangent to the fibers of $\pi$. 

If $\Omega^k$ is the sheaf of $\CC$-valued smooth $k$-forms on $X$, we have the de Rham complex of $X$: 

\begin{equation*}
K^\bullet: \qquad \qquad \ldots \xrightarrow{d} \Omega^k(X) \xrightarrow{d} \Omega^{k+1}(X) \xrightarrow{d} \ldots
\end{equation*}
which is endowed with the following descending filtration:
\begin{equation}\label{filtrarea}
F^pK^{p+q}:= \{\omega \in \Omega^{p+q}(X) \mid \iota_{X_{q+1}}\ldots\iota_{X_1}\omega =0, \forall X_1, \ldots, X_{q+1} \in \mathcal{X}^v(X)\}.
\end{equation}
By the theory of spectral sequences, this filtration determines a sequence of double complexes $(E_r^{\bullet,\bullet},d_r)_{r\geq0}$ with $d_r:E_r^{p,q}\rightarrow E^{p+r,q-r+1}_r$ of bidegree $(r,-r+1)$, which computes the cohomology of the complex $(K^\bullet,d)$. More precisely, if we denote by $E^{p,q}_\infty:=\lim_{r\to\infty}E_r^{p,q}$, we have:
\begin{equation*}
H^k(X,\CC):=H^k(K^\bullet)=\oplus_{p+q=k}E^{p,q}_\infty \ \ \ 0\leq k\leq \dim_\RR X.
\end{equation*}

The complex $E_{r+1}$, called the $(r+1)$-th page of $E$, is defined recurrently as the cohomology of $(E_r,d_r)$.
We now make explicit the definition of each page of the spectral sequence. The  pages $E_0$ and $E_1$ are simply given by:
\begin{align*}
E_0^{p, q} = \frac{F^pK^{p+q}}{F^{p+1}K^{p+q}}, \ \ \ &d_0 : E_0^{p, q} \rightarrow E_0^{p, q+1}\\
&d_0(\hat{\eta}_{p, q}) = \widehat{d\eta}_{p, q+1}\\
E_1^{p, q} = \frac{\Ker \, d^{p, q}_0}{\Im \, d^{p, q-1}_0}, \ \ \ &d_1: E^{p, q}_1 \rightarrow E^{p+1, q}_1\\
&d_1([\widehat{\eta}]^{p, q}_{d_0})=[\widehat{d\eta}]^{p+1, q}_{d_0}.
\end{align*}
The second page is again:
\begin{equation*}
E_2^{p, q} = \frac{\Ker \, d^{p, q}_1}{\Im \, d_1^{p-1, q}}, \ \ \ d_2 : E_2^{p, q} \rightarrow E_2^{p+2, q-1}.\end{equation*}
In order to write down $d_2$, one needs now to make sense of the objects of $E_2$. If $[\hat{\eta}]^{p, q}_{d_0}\in \Ker \ d_1^{p,q}$, then there exists $\hat{\xi} \in E_0^{p+1, q-1}$ such that $\widehat{d\eta} = d_0\hat{\xi} =\widehat{d\xi}$, hence $\widehat{d\eta-d\xi}=0$, meaning that $d\eta-d\xi \in F^{p+2}K^{p+q+1}$. Then:
\begin{equation}\label{D2}
d_2([[\hat{\eta}]_{d_0}]_{d_1}) = [[\widehat{d\eta-d\xi}]_{d_0}]_{d_1}.
\end{equation}
In general, by induction, one can show that $d_r:E_r^{p, q} \rightarrow E_r^{p+r, q-r+1}$ is given by:
\begin{equation}\label{diffr}
d_r([\ldots[[\hat{\eta}]_{d_0}]_{d_1}\ldots]_{d_{r-1}})=[\ldots[[\widehat{d\eta - d\delta}]_{d_0}]_{d_1}\ldots]_{d_{r-1}},
\end{equation}
where $\delta = \xi_1+\ldots+\xi_{r-1}$, and the elements $\xi_1 \in F^{p+1}K^{p+q}, \ldots, \xi_{r-1} \in F^{p+r-1}K^{p+q}$ are chosen via diagram chasing  such that $d\eta - d\delta \in F^{p+r}K^{p+q+1}$. 

When the fiber bundle $\pi$ is flat, one has a $\ce$ splitting of $TX=T_B\oplus T_F$, where, locally, $T_B$ is the tangent space of $B$ and $T_F$, of the fiber. The differential $d$ also splits as $d_B+d_F$ with $d_B^2=0$ and $d_F^2=0$, where $d_B$ is the derivation in the direction of the basis and $d_F$ is the derivation along the fiber. In this case, the spectral sequence becomes more explicit. We have an induced splitting of the vector bundle of $k$-forms:
\begin{equation*}
\textstyle\bigwedge^kT^*X=\oplus_{p+q=k}\textstyle\bigwedge^p T_B^*\otimes\textstyle\bigwedge^q T_F^*
\end{equation*}
and then $E_0^{p,q}=\ce(X,\bigwedge^pT_B^*\otimes\bigwedge^q T_F^*)$ for any $0 \leq p, q \leq k$. Moreover, one has:
\begin{align*}
d_0=d_F,  \ \  d_1([\al]_{d_F})=[d_B\al]_{d_F}.
\end{align*}
For $a\in\Ker\ d_1^{p,q}$ represented by $\eta\in\Ker\ d_F\subset E_0^{p,q}$, there exists $\xi\in E_0^{p+1,q-1}$ so that $d_B\eta=d_F\xi$. One then has: 
\begin{equation*}
d\eta-d\xi=d_F\eta+d_B\eta-d_F\xi-d_B\xi=-d_B\xi\in\ker d_F\subset E_0^{p+2,q-1}
\end{equation*}
so that, by \eqref{D2}, $d_2$ is given by:
\begin{equation*}
d_2([[\eta]_{d_F}]_{d_B})=-[[d_B\xi]_{d_F}]_{d_B}.
\end{equation*}
In general, for a given element $[\ldots[[\eta]_{d_F}]_{d_B}\ldots]_{d_{r-1}}\in E_r^{p,q}$, \eqref{diffr} tells us that: 
\begin{equation*}
d_r([\ldots[[\eta]_{d_F}]_{d_B}\ldots]_{d_{r-1}}) = - [\ldots[[d_B \xi_{r-1}]_{d_F}]\ldots]_{d_{r-1}}
\end{equation*}
for some $\xi_{r-1} \in E_{0}^{p+r-1, q-r+1}$ such that  there exist $\xi_1\in E_{0}^{p+1, q-1}, \ldots, \xi_{r-2}\in E_{0}^{p+r-2, q-r+2}$ that satisfy $d\eta-d\xi_1-\ldots d\xi_{r-1} \in E_0^{p+r, q-r+1}$, obtained by chasing diagrams.

\subsection{ Twisted  cohomology}
Let $M$ be a compact differentiable manifold, let $\theta$ be a complex valued closed one-form on $M$ and let $d_\theta$ be the differential operator $d_\theta=d-\theta\wedge\cdot$. Since $d_\theta^2=0$, we have a complex:
\begin{equation*}
\xymatrix{
0\ar[r] & \Omega^0_{M}(M)\ar[r]^-{d_\theta} & \Omega^1_{M}(M) \ar[r]^-{d_\theta}&\ldots
}
\end{equation*}
whose cohomology $H_\theta^\bullet(M):=\frac{\Ker d_\theta}{\Im d_\theta}$ is called the \textit{twisted} cohomology associated to $[\theta]_{dR}$. Indeed, it depends only on the de Rham cohomology class of $\theta$. If $M$ is orientable, there is a version of Poincar\' e duality that holds for $H^\bullet_\theta(M)$, see for instance \cite[Corollary~3.3.12]{dimca}, and that is: $H^\bullet_\theta(M)^* \cong H^{N-\bullet}_{-\theta}(M)$, where $N=\mathrm{dim}_{\R}M$. Moreover, we have the following result:

\begin{lemma}\label{h0}
Let $\tau\in H^1(M,\CC)$ be a de Rham class. Then the following are equivalent:
\begin{enumerate}
\item $H^0_\tau(M)\neq 0$;
\item $\tau\in H^1(M,2\pi i\ZZ)\subset H^1(M,\CC)$;
\item For any $k\in\ZZ$, $H^k_\tau(M,\CC)\cong H^k_{dR}(M,\CC)$.
\end{enumerate}
\end{lemma}
\begin{proof}
Clearly, if $(3)$ holds then $H^0_\tau(M,\CC)=H^0(M,\CC)\neq 0$, so we have $(1)$.

Now suppose $(1)$ holds, meaning that if we choose a representative $\theta\in\tau$, there exists a smooth function $h:M\rightarrow \CC$ so that $h\theta=dh$, with $h$ not identically zero. Then we also have $\ov h\ov{\theta}=d\ov h$, which implies $d|h|^2=|h|^22\Re\theta$. This is a linear first order differential system, so if $|h|^2$ has some zero, then $h$ would vanish everywhere on $M$. Thus, we have $2\Re\theta=d\ln|h|^{2}$, and without any loss of generality, we can now suppose that $\Re\theta=0$.

On the universal cover $\tilde M$, there exists $f\in\ce(\tilde M,\CC)$ so that $\theta=df$. Then we find:
\begin{equation*} 
d(\e^{-f}h)=\e^{-f}(-dfh+dh)=0
\end{equation*}
thus $h=c\e^f$, with $c\in\CC$ a constant. On the other hand, by the universal coefficient theorem and Hurewicz theorem we have $H^1(M,\CC)\cong \Hom(\pi_1(M),\CC)$, and the homomorphism $\tau:\pi_1(M)\rightarrow \CC$ correspondig to $\theta$ is precisely given by $\tau(\gamma)=\gamma^*f-f$, $\gamma\in\pi_1(M)$. Thus, as $\e^f=c^{-1}h$ is defined on $M$, it is $\pi_1(M)$-invariant as a function on $\tilde M$, so that we have
\begin{equation*}
\gamma^*\e^f=\e^f\e^{\tau(\gamma)}=\e^f, \ \ \forall\gamma\in\pi_1(M)
\end{equation*}
implying that $\Im\tau\subset 2\pi i\ZZ$, from which assertion $(2)$ follows.

Similarly, if $(2)$ holds and we choose $\theta$ a representative of $\tau$ and write $\theta=df$ on $\tilde M$, then as $\tau(\gamma)=\gamma^*f-f\in 2\pi i\ZZ$ for any $\gamma\in\pi_1(M)$, the function $h=\e^f$ is $\pi_1(M)$-invariant and descends to a well-defined function $h:M\rightarrow \CC^*$ satisfying $dh=h\theta$. Finally, let us note that in this case $d_\theta(\cdot)=hd(h^{-1}\cdot)$, which establishes an isomorphism between the twisted cohomology $H^\bullet_\tau(M)$ and $H^\bullet(M,\CC)$.
\end{proof}

\begin{remark}
A result of \cite{llmp} states that if $\theta\in\Omega_M^{1}(M,\RR)$ is a non-zero closed form, and there exists a Riemannian metric on $M$ so that $\theta$ is parallel for the corresponding Levi-Civita connection, then we have $H^\bullet_\theta(M)=0$. Note that this is not true if $\theta$ is complex valued.
\end{remark}

The twisted cohomology can also be seen as the cohomology of certain flat line bundles. In general, these are parametrized by elements $\rho\in\Hom(\pi_1(M),\CC^*)$ as follows: we let $L_\rho$ be the induced complex line bundle over $M$, that is the quotient of $\tilde M\times\CC$ by the action of $\pi_1(M)$ given by: 
\begin{equation*}
\gamma(x,\la)=(\gamma(x),\rho(\gamma)\la), \ \ \ \gamma\in \pi_1(M), \ \ (x,\la)\in\tilde M\times\CC.
\end{equation*}
Moreover, we endow $L_\rho$ with the unique flat connection $\nabla$ whose corresponding parallel sections are exactly the locally constant sections of $L_\rho$. Denote by $d^\nabla$ the differential operator acting on $\Omega_M^\bullet\otimes L_\tau$ which is induced by $\nabla$ by the Leibniz rule. Then the cohomology of ($L_\rho ,\nabla)$ is the cohomology denoted by $H^\bullet(M, L_\rho )$ of the complex:
\begin{equation*}
\xymatrix{
0\ar[r] & \Omega^0_{M}(M,L_{\rho }) \ar[r]^-{d^\nabla} &\Omega^1_{M}(M, L_{\rho }) \ar[r]^-{d^\nabla} &\ldots.
}
\end{equation*}

Equivalently, if we let $\mathcal{L}_\rho$ be the sheaf of parallel sections of $(L_\rho,\nabla)$, then we also have a natural isomorphism $H^\bullet(M, L_\rho)\cong H^\bullet(M,\mathcal{L}_\rho)$, where the latter is the sheaf cohomology. $\mathcal{L_\rho}$ is called a local system, and determines and is completely determined by $(L_\rho,\nabla)$.

On the other hand, the exponential induces an exact sequence:
\begin{equation}\label{exp}
 \xymatrix{
 0\ar[r] &H^1(M,2\pi i\ZZ)\ar[r] & H^1(M,\CC) \ar[r]^-{\exp} & H^1(M,\CC^*)\ar[r]^-{c_1} & H^2(M,\ZZ)}
\end{equation}
and all elements $\rho\in\ker c_1\subset H^1(M,\CC^*)\cong\Hom(\pi_1(M),\CC^*)$ are of the form $\exp\tau$, with $\tau\in\Hom(\pi_1(M),\CC)$. For the corresponding flat line bundle $(L_\rho,\nabla)$, the connection has an explicit form. We choose $\theta\in\tau$ a representative, write $\theta=d\phi$ on $\tilde M$, so that $s=\e^\phi$ determines a global trivialising section of $L_\rho$. Then $\nabla$ is given by $\nabla s=\theta\otimes s$, and it can be easily seen that this construction does not depend on the chosen $\tau\in\exp^{-1}(\rho)$, nor on $\theta\in\tau$. Moreover, we have a natural isomorphism:
 \begin{equation*}
H_\theta^\bullet(M)\cong H^\bullet(M,L_\rho^*).
\end{equation*}
Also note that if $H^2(M,\ZZ)$, or also $H_1(M,\ZZ)$, has no torsion, then the map $c_1$ in \eqref{exp} is zero, and then all flat line bundles on $M$ are of this form.

In particular, by \ref{h0} we have the following result:

\begin{lemma}\label{S1} $H^\bullet_\tau(\Ss^1, \CC)=0$ if and only if $\tau \notin H^1(\Ss^1, 2\pi i\ZZ)$.
\end{lemma}
\begin{proof}
The only if part is assured by \ref{h0}. On the other hand, if $\tau\notin H^1(\Ss^1,2\pi i\ZZ)$, then also $-\tau\notin H^1(\Ss^1,2\pi i\ZZ)$, thus \ref{h0} implies $H^0_\tau(\Ss^1)=H^0_{-\tau}(\Ss^1)=0$. Finally, by Poincar\'e duality we find $H^1_{\tau}(\Ss^1)\cong H^0_{-\tau}(\Ss^1)^*=0$, which concludes the proof. \end{proof}

This allows us to prove the following, which we will use a number of times in the sequel:
 
\begin{lemma}\label{isom}
Let $\TT^s$ be the compact $s$-dimensional torus, let $\rho:\pi_1(\TT^s)\rightarrow \CC^*$ be any representation of $\pi_1(\TT^s)$ on $\CC$ and let $(L_\rho,\nabla)\rightarrow \TT^s$ be the associated flat complex line bundle. Then $H^\bullet(\TT^s,L_\rho)=0$ if and only if $\rho$ is not trivial. 
\end{lemma}
\begin{proof}
Let $\rho\in\Hom(\pi_1(M),\CC^*)$ be a non-trivial element. As $H^2(\TT^s,\ZZ)$ is a free abelian group, there exists $0 \neq \tau\in H^1(\TT^s,\CC)$ so that $\rho=\exp\tau$. We write then $L_\rho=L_\tau$.

Let us identify $\TT^s$ with $(\Ss^1)^s$, and let $p_k:\TT^s\rightarrow \Ss^1$ the the projection on the $k$-th component, for $k\in\{1,\ldots, s\}$. If we denote by $\nu$ a generator of $H^1(\Ss^1,\ZZ)$, then $\tau$ writes $\tau=\sum_{k=1}^sa_kp_k^*\nu$, with $a_1,\ldots, a_s\in\CC$, not all $2\pi i\ZZ$-valued. In particular, it follows that 
\begin{equation*}
L_\tau\cong p_1^*L_{a_1\nu}\otimes\ldots\otimes p_s^*L_{a_s\nu}.
\end{equation*} 
Now, by the  K\" unneth formula for local systems (see \cite[Corollary~2.3.31]{dimca}) it follows that:
\begin{equation*}
H^\bullet(\TT^s,L_\tau)\cong H^\bullet(\Ss^1,L_{a_1\nu})\otimes\cdots\otimes H^\bullet(\Ss^1,L_{a_s\nu}).
\end{equation*}
Since there exists at least one $k\in\{1,\ldots, s\}$ with $a_k\notin 2\pi i\ZZ$, \ref{S1} implies that $H^\bullet(\Ss^1,L_{a_k\nu)}$ vanishes, and the conclusion follows.\end{proof}

\subsection*{Notation} In all that follows, the sheaf of complex valued $\ce$ $l$-forms on $X$ will be denoted by $\Omega_X^l$ or simply by $\Omega^l$, if there is no ambiguity about the manifold $X$, and its global sections will be denoted by $\Omega_X^l(X)$ or by $\Omega^l(X)$.
Also, for a given OT manifold $X$ of type $(s,t)$  corresponding to $(K,U)$, we will sometimes denote by $\Gamma:=U\ltimes O_K$ its fundamental group and by $\hat X:=\HH^s\times\CC^t/O_K$. Concerning the compact tori that will appear in our discussion, we will use the notation $\TT^k$ for the $k$-dimensional torus viewed as a smooth manifold (without any additional structure), and $\TT$ for the $n$-dimensional abelian compact Lie group which, in our case, acts on $\hat X$, where $n=2t+s$.  As already mentioned, $U$ acts on $\TT$ by conjugation, and for any $u\in U$, we will denote by $c_u\in\Aut(\TT)$ the automorphism $c_u(a)=u^{-1}au$. For any $q\in\NN^*$ we will denote by $\mathcal{I}_q$ the set of multi-indexes $I=(0<i_1<\ldots<i_q\leq n)$ and for $I\in\mathcal{I}_q$ we will denote by $|I|$ the length of $I$ which is $q$. Finally, for a given representation $\rho:\pi_1(\TT^s)=U_\HH\rightarrow\CC^*$, we denote by $L_\rho$ the induced flat complex line bundle over $\TT^s$, and for a closed one-form $\theta$ on $X$, we denote by $\rho^\theta\in\Hom(U,\CC^*)=\Hom(\Gamma,\CC^*)$ the representation it induces.

\section{The de Rham cohomology}

In the next two sections, we will prove in two different ways the following result:

\begin{theorem}\label{teoremaA}
Let $X=X(K,U)$ be an OT manifold of type $(s,t)$ of complex dimension $m$. For any $l\in\{0,\ldots, 2m\}$ we have:
\begin{equation*}
H^l(X,\CC)\cong\displaystyle\bigoplus_{\substack{p+q=l\\|I|=q\\ \sigma_I=1}}\textstyle\bigwedge^p\CC\{d\ln \Im w_1,\ldots,d\ln \Im w_s\} \wedge e_I.
\end{equation*}

In particular, the Betti numbers of $X$ are given by:
\begin{equation*}
b_l=\sum_{p+q=l} {s\choose p}\cdot \rho_q, 
\end{equation*}
where $\rho_q$ is the cardinal of the set $\{I \mid |I|=q, \sigma_I = 1\}$.
\end{theorem} 

In this section, we compute the de Rham cohomology of an OT manifold $X$ by identifying it with the cohomology of invariant forms on $X$ with respect to a certain compact torus action. In order to be precise, let us fix an OT manifold $X=X(K,U)$ of type $(s,t)$ and of complex dimension $m=s+t$. Recall that $\TT=\TT^n$ acts holomorphically by translations on $\hat X=\HH^s\times\CC^t/O_K$, but not on $X$. However, by identifying smooth forms on $X$ with smooth $U$-invariant forms on $\hat X$, it makes then sense to speak of $\TT$-invariant forms on $X$: these will be exactly the $U\ltimes\TT$-invariant forms on $\hat X$. Let us denote by $A^\bullet$ the graded sheaf of such invariant forms, which is a subsheaf of $\Omega_X^\bullet$. The differential $d$ acting on $\Omega^\bullet_X$ fixes $A^\bullet$, so $(A^\bullet,d)$ is a subcomplex of the de Rham complex of $X$. As in the usual setting of a manifold endowed with a compact group action, we have the following:

\begin{lemma}
There exists a projection graded morphism $\pi:\Omega^\bullet_X\rightarrow A^\bullet$ commuting with the differential $d$. 
\end{lemma}
\begin{proof}
The projection morphism will be given by averaging over the torus action. Let us fix $0\leq l\leq 2m$, and consider a (local) smooth $l$-form $\eta$ on $X$, identified with a $U$-invariant form on $\hat X$. Let $\mu$ be the $\TT$-invariant $n$-form on $\TT$ with $\int_\TT\mu=1$, and let:
\begin{equation*}
\pi\eta:=\int_\TT a^*\eta \mu(a).
\end{equation*}
Clearly, $\pi\eta$ is a $\TT$-invariant form on $\hat X$. In order to see that it descends to a form on $X$, we have to show that $\pi\eta$ is $U$-invariant. Indeed, for any $u\in U$, if $c_u\in\Aut(\TT)$ is the conjugation $a\mapsto u^{-1}au$ as before, then we have:
\begin{align*}
u^*(\pi\eta)&=\int_\TT(au)^*\eta \mu(a)=\int_\TT(uc_u(a))^*\eta \mu(a)=\\
                    &=\int_\TT c_u(a)^*\eta \mu(a)=\int_\TT a^*\eta \mu(c_u^{-1}(a))=\\
                    &=\int_\TT a^*\eta \mu(a)=\pi\eta.
\end{align*} 
Above, we made the change of variable $a\mapsto c_u^{-1}(a)$, and then used the fact that $\mu$ is a constant $c_u$-invariant form on $\TT$, so that $\mu(c_u^{-1}(a))=((c_u^{-1})^*\mu)(a)=\mu(a)$. 

Finally, it is clear from the definition of $\pi$ that it commutes with $d$, and that $\pi$ restricted to $A^\bullet$ is the identity.
\end{proof}

In the context of a manifold endowed with a compact group action, a standard result states that the de Rham cohomology of the manifold is the cohomology of invariant forms. This is still true in our context, and the proof follows the same lines:

\begin{lemma}\label{Gpoincare}
For any $0\leq l\leq 2m$, any open set $O\subset X$  and any $d$-closed form $\eta\in\Omega^l_X(O)$ there exists some $\be\in \Omega^{l-1}_X(O)$ so that $\pi\eta-\eta=d\beta$. In particular, we have an isomorphism $H^l[\pi]:H^l(X,\CC)\rightarrow H^l(X,A^\bullet(X))$.
\end{lemma}

\begin{proof}
Let $\eta\in\Omega^l_X(O)$ be a closed form, let $\hat O$ be the preimage of $O$ in $\hat X$ and, as before, identify $\eta$ with a form on $\hat O$. Let $a\in\TT$ and let $\{\Phi_a^v\}_{v\in\RR}$ be a one-parameter subgroup of $\TT$ with $\Phi_a^1=a$. Let $\xi_a$ be the vector field on $\hat X$ generated by $\Phi_a$, and consider the map $F_a:\RR\times \hat O\rightarrow \hat O$, $(v,x)\mapsto \Phi_a^v(x)=x+va$. We then have $F_a^*\eta=\eta_1+dv\wedge\eta_2$, with $\eta_1(v,\cdot)=(\Phi_a^v)^*\eta$ and $\eta_2(v,\cdot)=\iota_{\xi_a}(\Phi_a^v)^*\eta$. If we denote by $d_X$ the differential with respect to the $X$-variables on $\RR\times\hat O$, then $d\eta=0$ implies:
\begin{equation*}
0=dF_a^*\eta=d_X\eta_1+dv\wedge\pa{v}\eta_1-dv\wedge d_X\eta_2.
\end{equation*}
In particular, we have $\pa{v}\eta_1=d_X\eta_2$, or also:
\begin{equation*}
a^*\eta-\eta=\int_0^1\pa{v}\eta_1dv=\int_0^1d_X\eta_2dv=d\int_0^1\eta_2dv.
\end{equation*}
Denoting by $\beta_a$ the form $\int_0^1\eta_2dv$ and by $\beta:=\int_T\beta_a\mu(a)$, we have $\pi\eta-\eta=d\beta$, and we are then left with showing that the form $\beta$ is $U$-invariant. Let $u\in U$ and $a\in\TT$. Upon noting that 
\begin{align*}
u^{-1}_*\xi_a&=\frac{d}{dv}|_{v=0}(u^{-1}\Phi^v_a)=\frac{d}{dv}|_{v=0}(\Phi^v_{c_u(a)}u^{-1})=\xi_{c_u(a)}
\end{align*}
we have:
\begin{align*}
u^*\be_a&=\int_0^1u^*\iota_{\xi_a}(\Phi^v_a)^*\eta dv=\int_0^1\iota_{u^{-1}_*\xi_a}u^*(\Phi_a^v)^*\eta dv\\
&=\int_0^1\iota_{\xi_{c_u(a)}}(u\Phi^v_{c_u(a)})^*\eta dv=\int_0^1\iota_{\xi_{c_u(a)}}(\Phi^v_{c_u(a)})^*\eta dv=\beta_{c_u(a)}.
\end{align*}
So, as in the previous lemma, the $\TT$-invariance of $\mu$ implies then that $\int_\TT\beta_{c_u(a)}\mu(a)=\beta$.

For the last assertion, it is enough to see that the inclusion $\iota:A^\bullet\rightarrow \Omega^\bullet$ induces an isomorphism $H(\iota)$ in cohomology. If $\eta\in A^l(X)$ verifies $\eta=d\al$, $\al\in\Omega^{l-1}(X)$, then $\eta=\pi\eta=d\pi\al$, so $H(\iota)$ is injective. If $\eta\in\Omega^l$ is a closed form, then by the above we have $H(\iota)[\pi\eta]=[\pi\eta]_{dR}=[\eta]_{dR}$, so $H(\iota)$ is surjective. 
\end{proof}

For the sequel, we will fix some $l\in\{0,\ldots, 2m\}$. Recalling that the tangent bundle of $\hat X\cong (\RR_{>0})^s\times\TT^n$ splits smoothly as  $T\hat X=E\oplus V$, where $E$ is the pullback of $T(\RR_{>0})^s$ and $V$ is the pullback of $T\TT^n$ on $\hat X$, we have:
\begin{equation*}
\textstyle\bigwedge^lT^*\hat X=\oplus_{p=0}^l\textstyle\bigwedge^p E^*\otimes\textstyle\bigwedge^{l-p}V^*.
\end{equation*}
If we denote by $A^l_p$ the sheaf which associates to any open set $O\subset X$ 
\begin{equation*}
A^l_p(O):=A^l(O)\cap \ce(\hat O,\textstyle\bigwedge^p E^*\otimes\textstyle\bigwedge^{l-p}V^*\otimes\CC),
\end{equation*}
  where $\hat O$ is the pre-image of $O$ in $\hat X$, then we also have:
\begin{equation}\label{grad}
A^l = \oplus_{p=0}^l A^l_p.
\end{equation} 
At the same time, $A^l$ can be seen as: 
\begin{equation*}
A^l(O)=\{\eta\in \Omega^l_{\hat X}(\hat O)|\eta \text{ is }U \text{ invariant and } \LL_Z\eta=0 \ \ \forall Z\in\ce(\hat O, V)\}
\end{equation*} 
which implies that the differential $d$  is compatible with the grading of $A^\bullet$ given by \eqref{grad}, in the sense that $d(A^l_p)\subset A^{l+1}_{p+1}$ for any $0\leq p\leq l$. 

Hence, a form $\eta=\sum_{p=0}^l\eta_p\in A^l(X)$ splitted with respect to the grading \eqref{grad} is closed if and only if each $\eta_p$ is closed. As a consequence, the complex $0\rightarrow A^\bullet(X)$ splits in the subcomplexes:
\begin{equation}
\xymatrix{
C_p^\bullet:  \qquad 0 \ar[r]^-d & A_0^p(X) \ar[r]^-d & A_1^{p+1}(X) \ar[r]^-d & A_2^{p+2}(X)\ar[r]^-d &\ldots }
\end{equation}

for $0\leq p\leq s$. Moreover, if a form $\eta=\sum_{p=0}^l\eta_p\in A^l(X)$ is exact: $\eta=d\beta$, then writing again $\be=\sum_{q=0}^{l-1}\be_q$, we must have $\eta_{p+1}=d\be_p$ for any $0\leq p\leq l-1$ and $\eta_0=0$. So we see that $\eta$ is exact if and only if each $\eta_p$ is. Hence, if we let:
\begin{equation*}
H^l_p(X,A):=\frac{\ker d:A^l_p(X)\rightarrow A^{l+1}_{p+1}(X)}{\Im d:A^{l-1}_{p-1}(X)\rightarrow A^l_p(X)}=H^p(C^\bullet_{l-p})
\end{equation*}
then we have:
\begin{equation}\label{split1}
H^l(X,A^\bullet(X))=\oplus_{p=0}^l H^l_p(X,A).
\end{equation}

Now let us take a closer look at the complex $C_l^\bullet$. Denoting by $d_E$ the differentiation in the $E$ direction, $(C^\bullet_l,d)$ is a subcomplex of:
\begin{equation*}
(\ce(\hat X,\textstyle\bigwedge^lV^*\otimes\textstyle\bigwedge^\bullet E^*\otimes\CC), d_E)
\end{equation*}
which, in turn, is just $\ce(\hat X,\bigwedge^l V^*\otimes\CC)$ tensorized by:
\begin{equation*}
\xymatrix{
0 \ar[r] & \ce(\hat X, \CC) \ar[r]^-{d_E} &\ce(\hat X,  E^*\otimes\CC) \ar[r]^-{d_E} & \ce(\hat X, \bigwedge^2 E^*\otimes\CC)\ar[r] &\ldots }
\end{equation*}

But recall that, for any $0\leq q\leq l$, $\bigwedge^qV^*\otimes\CC$ is globally trivialised over $\hat X$ by $\{e_I\}_{I\in\mathcal {I}_q}$, where $\mathcal{I}_q$ denotes the set of all multi-indexes $I=(0<i_1<\ldots<i_q\leq n)$ and the forms $e_I$ were defined in \eqref{frame}. Thus, for any $0\leq p\leq l$, we have:
\begin{equation*}
\ce(\hat X, \textstyle\bigwedge^qV^*\otimes\textstyle\bigwedge^pE^*\otimes\CC)=\oplus_{I\in\mathcal{I}_q}\ce(\hat X, \textstyle\bigwedge^p E^*)\otimes\CC e_I.
\end{equation*} 
Moreover, a section $\eta=f\otimes e_I$ of $\bigwedge^p E^*\otimes \CC e_I$ belongs to $A_p^{p+|I|}$ if and only if it is $\TT$-invariant and
\begin{equation}\label{equiv}
u^*f=\sigma_I(u)^{-1}f \text{ for any }u\in U.
\end{equation}
If we denote by $E^p_{\sigma_I}$ the sheaf of $\TT$-invariant sections $f$ of $\bigwedge^pE^*\otimes \CC$ which are $\sigma_I^{-1}$ equivariant, i.e. verify \eqref{equiv}, it follows that we have:
\begin{equation*}
A^{q+p}_p=\oplus_{I\in\mathcal{I}_q} E^p_{\sigma_I}\wedge e_I.
\end{equation*}
Moreover, as the $e_I$'s are closed forms, we have: 
\begin{equation*}
d: E^p_{\sigma_I}\wedge e_I\rightarrow E^{p+1}_{\sigma_I}\wedge e_I \ \ \ \ \forall I\in\mathcal{I}_q.
\end{equation*} 
So finally we get that the complex $C^\bullet_{l-p}$ splits into the complexes on $X$:
\begin{equation}\label{cohE}
\xymatrix{
C_{l-p}^\bullet(I):  \qquad 0 \ar[r] & E^0_{\sigma_I}(\hat X)\wedge e_I \ar[r]^-d & E^{1}_{\sigma_I}(\hat X)\wedge e_I \ar[r]^-d & E^2_{\sigma_I}(\hat X)\wedge e_I\ar[r]^-d &\ldots }
\end{equation}
indexed after all $I\in\mathcal{I}_{l-p}$. So also the cohomology splits:
\begin{equation}\label{split2}
H^l_p(X,A)=\oplus_{I\in\mathcal{I}_{l-p}} H^p_{\sigma_I}(X, A)
\end{equation}
where $H^p_{\sigma_I}(X, A):=H^p(C^\bullet_{l-p}(I))$.

At the same time, the $\TT$-invariant sections of $\bigwedge^pE^*\otimes \CC$ over $\hat X$ naturally identify with the sections of $\Omega^p_{(\RR_{>0})^s}$ over $(\RR_{>0})^s$. Hence, the sections of $E^p_{\sigma_I}$ coincide then with the sections of $\Omega^p_{\TT^s}\otimes L^*_{\sigma_I}$, and we have:
\begin{equation}\label{isocoh}
H^p_{\sigma_I}(X,A)\cong H^p(\TT^s,L^*_{\sigma_I})\otimes e_I. 
\end{equation}
So, putting together \eqref{isocoh}, \eqref{split2}, \eqref{split1}, \ref{Gpoincare} and \ref{isom}, we get: 
\begin{equation*}
H^l(X,\CC)\cong\oplus_{p+q=l}\oplus_{I\in\mathcal{I}_q}H^p(\TT^s,L^*_{\sigma_I})\otimes e_I
\end{equation*}
leading, together with \ref{isom}, to \ref{teoremaA}.

\section{The Leray-Serre spectral sequence of OT manifolds}

Let $X=X(K,U)$ be an OT manifold of type $(s,t)$. In this section, we are interested in computing its de Rham cohomology using the Leray-Serre spectral sequence associated to the fibration depicted in \eqref{fiberbundle}:
\begin{equation*}
\mathbb{T}^n \rightarrow X \xrightarrow{\pi} \mathbb{T}^s.
\end{equation*}
We endow the de Rham complex of $X$ with the filtration described in \eqref{filtrarea}. It turns out that the Leray-Serre sequence associated to this filtration degenerates at the page $E_2$ and we prove this by outlining the special properties of the OT fiber bundle. 

Let us start by noting that we have two fiber bundles over $\TT^s$ associated to this fibration:
\begin{equation}\label{fibrati1}
\Omega^\bullet(\mathbb{T}^n) \rightarrow {\bf{\Omega^\bullet(\mathbb{T}^n)}} \rightarrow \mathbb{T}^s
\end{equation}
\begin{equation}\label{fibrati2}
H^\bullet(\mathbb{T}^n,\CC) \rightarrow {\bf {H^\bullet(\mathbb{T}^n)}} \rightarrow \mathbb{T}^s.
\end{equation}
Indeed, recall that we have an action of $U_\HH$ on $\TT^n$ defined in \eqref{reprezentare}, with respect to which $\pi$ is defined as $(\RR_{>0})^s\times\TT^n/U_\HH\rightarrow\TT^s$. But then we also have an induced action of $U_\HH$ on $\Omega^\bullet(\TT^n)$ by push-forward, which defines ${\bf{\Omega^\bullet(\TT^n)}}:=(\RR_{>0})^s\times\Omega^\bullet(\TT^n)/U_\HH\rightarrow \TT^s$ as an infinite-dimensional vector bundle over $\TT^s$. Also we have an induced action of $U_\HH$ on $H^\bullet(\TT^n,\CC)$, which then defines the vector bundle ${\bf{H^\bullet(\TT^n)}}:=(\RR_{>0})^s\times H^\bullet(\TT^n,\CC)/U_\HH\rightarrow\TT^s$. 

{\textit{Fact 1: The fibration is locally constant}}, meaning that if $U_\al \cap U_\be$ is a connected open subset of $\TT^s$, then $g_{\al\be}: U_\al \cap U_\be \times \mathbb{T}^n \rightarrow U_\al \cap U_\be \times \mathbb{T}^n$ only depends on the $\mathbb{T}^n$-variables. This allows us to make the following identification:
\begin{equation}\label{e0iso}
E_0^{p, q}\simeq  \Omega^p(\mathbb{T}^s, {\bf{\Omega^q(\mathbb{T}^n)}}).
\end{equation}
Indeed, recall that we have $TX=E\oplus V$, where, locally, $E$ is the tangent bundle of the base $\TT^s$ and $V$ is the tangent bundle of the fiber $\TT^n$, and we have identified $E_0^{p,q}$ with $\ce(X,\bigwedge^pE^*\otimes\bigwedge^qV^*\otimes\CC)$. Consider $\eta \in E_0^{p,q}$ and suppose that $U_\al$ is an open set of $\TT^s$ trivializing $\pi$ via $\phi_\al:\pi^{-1}(U_\al)\rightarrow U_\al\times\TT^n$. Write $(\phi_\al)_* \eta = \sum_i a_{i}^\al\wedge b^\al_i$, where, for each $i$, $a^\al_i$ is a $p$-form on $U_\al$ and $b^\al_i$ is an element of $\ce(U_\al\times\TT^n,\bigwedge^qT^*\TT^n\otimes\CC)$ which may depend on both the coordinates of $U_\al$ and of $\mathbb{T}^n$. Of course, the forms $a^\al_i$ and $b^\al_i$ are not unique. If $(U_\be,\phi_\be)$ is another trivializing open set for $\pi$ intersecting $U_\al$, then we have:
\begin{equation*}
(\phi_\be)_* \eta = (\phi_\be \circ \phi_\al^{-1})_* \circ (\phi_\al)_* \eta = (g_{\be\al})_*  \sum_i a^\al_i \wedge b^\al_i.
\end{equation*}
As $g_{\be\al}$ is locally constant on $U_\al\cap U_\be$, $(g_{\be\al})_*a_i^\al = a_i^\al$, therefore $(\phi_\be)_* \eta = \sum_i a^\al_i \wedge (g_{\be\al})_*b^\al_i$. In particular, for each $i$, the forms $\{a_i^\al\}_\al$ glue up to a well-defined global $p$-form $a_i$ on $\TT^s$ and $\eta$ is then an element of $\Omega^p(\TT^s, \bf{\Omega^q(\TT^n)})$.

 {\textit{Fact 2: ${\bf{H^q(\mathbb{T}^n)}}$ is a completely reducible local system.}} Indeed, as already mentioned, $\bf{H^q(\TT^n)}$ is a flat vector bundle defined by the induced representation $[R]:U_\HH\rightarrow\Aut(H^q(\TT^n,\CC))$. In order to determine $[R]$, recall that we have fixed a frame for $V^*$ over $(\RR_{>0})^s\times\TT^n$ given by $\{e_1,\ldots, e_n\}$ in \eqref{frame}. As this frame does not depend on $(\RR_{>0})^s$, it induces a frame for $T^*\TT^n$ over $\TT^n$ which we will denote the same, and we have $H^q(\TT^n,\CC)=\bigwedge^q\CC\{e_1,\ldots,e_n\}=\oplus_{I\in\mathcal{I}_q}\CC e_I.$ Then, for any $I$ and any $u\in U_\HH$, we have $[R](u)e_I=[R(u)_*e_I]=\sigma^{-1}_I(u)e_I$, or also $[R]=\sum_{I\in\mathcal{I}_q}\sigma^{-1}_I$ under the above direct sum decomposition. 

For any multi-index $I$, let us denote, as before, by $L_{\sigma_I}\rightarrow \TT^s$ the flat line bundle defined by the representation $\sigma_I$, so that ${\bf{H^q(\TT^n)}}=\oplus_{I\in\mathcal{I}_q}L^*_{\sigma_I}$. If $\nabla^I$ denotes the induced connection on $L^*_{\sigma_I}$ and $\nabla_q$ denotes the induced flat connection on $\bf{H^q(\TT^n)}$, then also $\nabla_q$ splits with respect to the direct sum decomposition as $\nabla_q=\sum_{I\in\mathcal{I}_q}\nabla^I$. In particular, we also obtain:
\begin{equation}\label{splitH}
H^p(\mathbb{T}^s, {\bf{H^q(\mathbb{T}^n)}})\cong\oplus_{I\in\mathcal{I}_q}H^p(\TT^s,L^*_{\sigma_I}).
\end{equation}

 {\textit {Fact 3: The base is a torus.}} This allows us to compute, via \ref{isom}:
\begin{equation}\label{EE}
H^p(\TT^s,{\bf{H^q(\TT^n)}})\cong\oplus_{\substack{I\in\mathcal{I}_{q}\\ \sigma_I\equiv 1}}H^p(\TT^s,\CC)\otimes e_I.
\end{equation}

Let us now describe the pages of the Leray-Serre sequence of the OT fibration.

{\textbf{Page 0}}: By Fact 1, we have $E_0^{p, q} \simeq \Omega^p(\mathbb{T}^s, {\bf{\Omega^q(\mathbb{T}^n)}})$.  In order to determine $d_0:E_0^{p,q}\rightarrow E_0^{p,q+1}$, which on $X$ corresponds to differentiation in the $V$-direction, let us first identify the corresponding operator on ${\bf{\Omega^\bullet(\TT^n)}}$. Consider the differential of $\TT^n$ which acts on $\Omega^\bullet(\TT^n)$, and then define by $d^v$ the operator acting on $(\RR_{>0})^s\times\Omega^\bullet(\TT^n)$ trivially on the first factor, and as the differential of $\TT^n$ on the second one. Clearly, this operator commutes with the action of $U_\HH$, and so descends to an operator $d^v$ on ${\bf{\Omega^\bullet(\TT^n)}}$. Under the isomorphism \eqref{e0iso}, we have then $d_0=d^v$, i.e. for $\hat{\eta}=\sum a_i\otimes b_i^\al\in\Omega^p(U_\al, {\bf{\Omega^q(\TT^n)}})$ we have:
\begin{equation}\label{D0}
 d_0(\sum a_i \otimes b^\al_i) = \sum (-1)^pa_i \otimes d^v b^\al_i.
 \end{equation}

{\bf{Page 1:}} By \eqref{D0} we have $E_1^{p, q} \simeq \Omega^p(\mathbb{T}^s, {\bf{H^q(\mathbb{T}^n)}}).$ The differential $d_1: \Omega^p(\mathbb{T}^s, {\bf{H^q(\mathbb{T}^n)}}) \rightarrow \Omega^{p+1}(\mathbb{T}^s, {\bf{H^q(\mathbb{T}^n)}})$ is identified then with $d_1=d^\nabla$, where $d^\nabla$ is the differential operator on $\TT^s$ induced by the flat connection $\nabla$ on ${\bf{H^q(\TT^n)}}$.

 {\bf{Page 2}}: From above, we deduce that $E_2^{p, q} \simeq H^p(\mathbb{T}^s, {\bf{H^q(\mathbb{T}^n)}}).$ Let $[[\eta]_{d_0}]_{d_1}\in E_2^{p,q}$ be a non-zero element and let $\eta= \sum a_i\otimes b_i$ locally. Then $[\eta]_{d_0}=\sum a_i\otimes[b_i]_{d^v}$. The fact that $[\eta]_{d_0}\in\ker d_1$ implies that there exists $\gamma\in  E_0^{p+1,q-1}$ so that $d^\nabla\sum(a_i\otimes b_i)=d^v\gamma$. As in Section 2.2, we have $(d^\nabla+d^v)(\sum a_i\otimes b_i-\gamma)=-d^\nabla\gamma\in\ker d^v\subset E_0^{p+2,q-1}$, hence, according to \eqref{D2}, $d_2$ is given by:
 \begin{equation}\label{D00}
d_2([\sum a_i\otimes[b_i]_{d^v}]_{d_1})=[[-d^\nabla\gamma]_{d^v}]_{d_1}. 
\end{equation}

At the same time, by \eqref{EE} in Fact 3, we have that any element $[[\eta]_{d^v}]_{d_1}$ of $E_2^{p,q}$ can be represented by a sum: 
\begin{equation}\label{eta}
\eta=\sum_{\substack{I\in\mathcal{I}_{q}\\ \sigma_I\equiv 1}}\al_I\otimes e_I\in E^{p,q}_0
\end{equation} 
where for each $I$ appearing in the sum, $\al_I\in\Omega^p(\TT^s)$ is a closed form on $\TT^s$, and $e_I$, given in \eqref{frame}, is $U$ invariant on $\hat X$, and so descends to a global element of $\bf{\Omega^q(\TT^n)}$ on $\TT^s$, verifying $d^\nabla e_I=0$. In particular, we have $d^\nabla\eta=0=d^v(0)$, so, by \eqref{D00} it follows that $d_2[[\eta]_{d^v}]_{d_1}=[[-d^\nabla 0]_{d^v}]_{d_1}=0$, so $d_2\equiv 0$. 

Finally, for any $r\geq 2$, any class in $E_r^{p,q}$ can be represented by $[\ldots[[\eta]_{d_0}]_{d_1}\ldots]_{d_{r-1}}$, where $\eta$ is of the form \eqref{eta}. Since $d\eta=(d^\nabla+d^v)\eta=0$, by \eqref{diffr} all $\xi_1, \ldots,\xi_{r-1}$ can be chosen to be zero, so $d_r\equiv 0$. Thus we have shown:

\begin{theorem}\label{teorema} The Leray-Serre spectral sequence of OT manifolds degenerates at $E_2$.
\end{theorem} 

As a corollary of this, one immediately obtains \ref{teoremaA}.

\section{Twisted cohomology of OT manifolds}

Now we want to compute the  twisted  cohomology groups of OT manifolds with respect to any closed one-form. We recall that $v_j$ stands for $\Im w_j$, for every $1\leq j \leq s$. The exact statement that we will obtain is the following:

\begin{theorem}\label{teoremaB} Let $X=X(K,U)$ be an OT manifold of type $(s,t)$ and of complex dimension $m$, and let $\theta=\sum_{k=1}^sa_k d\ln v_k$ be a closed one-form on $X(K, U)$, where $a_1,\ldots, a_s\in\CC$. Then for any $l\in\{0,\ldots, 2m\}$ we have:
\begin{equation*}
H^l_\theta(X,\CC)\cong\displaystyle\bigoplus_{\substack{p+q=l\\|I|=q\\ \rho^{\theta}\otimes\sigma_I=1}}\textstyle\bigwedge^p\CC\{d\ln v_1,\ldots,d\ln v_s\} \wedge(v_1^{a_1}\cdot \ldots \cdot v_s^{a_s})e_I.
\end{equation*}
In particular, the corresponding twisted Betti numbers are given by:
\begin{equation*}
b_l^\theta=\sum_{p+q=l} {s\choose p}\cdot \rho^\theta_q,
\end{equation*}
where $\rho^\theta_q$ is the cardinal of the set $\{I \mid |I|=q, \rho^{\theta}\otimes\sigma_I = 1\}$.
\end{theorem}

It is already known from \cite{ot} that $b_1(X)=s$, hence any closed one form is cohomologous to one of the form $\pi^*\eta$, where $\eta$ is closed one-form on $\mathbb{T}^s$. As the  twisted  cohomology $H^\bullet_{\theta}$ depends only on the de Rham cohomology class of $\theta$, and not on $\theta$ itself, we can assume that $\theta$ is the pullback of a form from $\mathbb{T}^s$.

We are going to use the same approach as in the previous section. Consider the complex:
\begin{equation*}
K_\theta^{\bullet}: \qquad \qquad \ldots \xrightarrow{d_\theta} \Omega^p(X) \xrightarrow{d_\theta} \Omega^{p+1}(X) \xrightarrow{d_\theta} \ldots
\end{equation*}
which we endow with the same descending filtration as before:
\begin{equation*}
F^pK_\theta^{p+q}:= \{\omega \in \Omega^{p+q}(X) \mid \iota_{X_{q+1}}\ldots\iota_{X_1}\omega =0, \forall X_1, X_2, \ldots X_{q+1} \in \mathcal{X}^v(X)\}.
\end{equation*}
It is easy to see that it is indeed a filtration, i.e. $d_{\theta}F^pK_\theta^{p+q} \subset F^{p}K_\theta^{p+q+1}$, as a consequence of $\theta$ being the pullback of a form from $\TT^s$. We study the spectral sequence associated to $K_\theta$ with this filtration, which we denote also by $E_{\bullet}$.

Again, we denote by ${\bf{\Omega^q(\mathbb{T}^n)}}$ and by ${\bf {H^q(\mathbb{T}^n)}}$ the vector bundles described in \eqref{fibrati1} and \eqref{fibrati2}, and as before we have the $0$-th page:
\begin{align*}
E_{0}^{p, q}=\frac{F^pK^{p+q}_\theta}{F^{p+1}K^{p+q}_\theta} \simeq \Omega^p(\mathbb{T}^s, {\bf{\Omega^q(\mathbb{T}^n)}})
\end{align*}
and via this isomorphism, $d_0:E_{0}^{p,q} \rightarrow E_{0}^{p, q+1}$ is given over a trivializing open set $U_\al$ by:
\begin{equation*}
d_0(\sum_i a_i \otimes b^\alpha_i) = \sum (-1)^p a_i \otimes d^v b^\alpha_i.
\end{equation*}
Thus, we again have:
\begin{equation*}
E_1^{p, q} \simeq \Omega^{p}(\mathbb{T}^s, {\bf {H^q(\mathbb{T}^n)}})
\end{equation*}
but this time, $d_1: E_1^{p, q} \rightarrow E_1^{p+1, q}$ is given over a trivializing open set $U_\al$ by:
\begin{equation*}
d_1(\sum a_i \otimes [b^\alpha_i]_{d^v}) =\sum d a_i \otimes [b^\alpha_i]_{d^v} + (-1)^p \sum a_i \wedge ([\nabla'b^\alpha_i]_{d^v} - \theta\otimes [b^\alpha_i]_{d^v})
\end{equation*}
where $\nabla'$ is the flat connection on ${\bf{H^q(\TT^n)}}$. Equivalently, if we see $\theta$ as a form on $\TT^s$ and define $L_\theta$ to be the complex flat line bundle over $\TT^s$ corresponding to $\exp[\theta]_{dR}\in H^1(\TT^s,\CC^*)\simeq \Hom(\pi_1(\TT^s),\CC^*)$, we have the following identification: 
\begin{align*}
E_1^{p,q} &\simeq \Omega^p(\mathbb{T}^s, L^{*}_\theta \otimes {\bf {H^q(\mathbb{T}^n)}})\\
d_1& =d^\nabla
\end{align*}
where $d^\nabla$ the differential operator induced by the corresponding flat connection on $L^{*}_\theta\otimes{\bf{H^q(\TT^n)}}$.
Thus we obtain the second page:
\begin{equation*}
E_2^{p, q} \simeq \frac{\Ker (d^{\nabla})^{p, q}}{\Im (d^{\nabla})^{p-1, q}} = H^p(\mathbb{T}^s, L^{*}_\theta \otimes \bf{H^q(\mathbb{T}^n)}).
\end{equation*}

Let $\theta=\sum_{k=1}^sa_k d\ln v_k$ with $a_k\in\CC$, which induces the representation $\rho^\theta=\sigma_1^{a_1}\otimes\cdots\otimes\sigma_s^{a_s}\in\Hom(\pi_1(\TT^s),\CC^*)$. 
The flat vector bundle $L^{*}_\theta\otimes{\bf{H^q(\TT^n)}}$ over $\TT^s$ is then given by the representation $[R]_\theta:U_\HH\rightarrow \Aut(H^q(\TT^n))$, $[R]_\theta:=(\rho^\theta)^{-1}\otimes[R]$.

We again have:
\begin{theorem}\label{propomain}
 The spectral sequence associated to $K^\bullet_\theta$ degenerates at the second page.
\end{theorem} 

\begin{proof} As before, we want to show that $d_r \equiv 0$ for $r \geq 2$. We notice that, as ${\bf {H^q(\mathbb{T}^n)}}$ is a completely reducible local system, then so is $L^{*}_\theta \otimes {\bf {H^q(\mathbb{T}^n)}}$. The same arguments as in Fact 2 and Fact 3 in Section 3 show that we have an isomorphism:
\begin{equation*}
H^p(\mathbb{T}^s, L^*_\theta \otimes {\bf {H^q(\mathbb{T}^n)}})\cong\oplus_{I\in\mathcal{I}_q}H^p(\mathbb{T}^s,  L^*_\theta\otimes L^*_{\sigma_I})\cong\oplus_{\substack{I\in\mathcal{I}_{q}\\ \rho^\theta\otimes\sigma_I\equiv 1}} H^p(\TT^s,\CC)\otimes e_I
\end{equation*} 
where  $e_I$ is now identified with a global parallel frame of $L^*_\theta\otimes L^*_{\sigma_I}$. This means that any element $[[\eta]_{d^v}]_{d^\nabla}\in E^{p,q}_2$ can be represented, globally on $\TT^s$, by: 
\begin{equation*}
\eta=\sum_{\substack{I\in\mathcal{I}_q\\ \rho^\theta\otimes\sigma_I\equiv 1}}\al_I\otimes e_I\in \Omega^p(\TT^s, L_\theta^* \otimes {\bf{\Omega^q(\TT^s)}}),
\end{equation*} 
with $\al_I$ closed one forms on $\TT^s$. Since we have $d^\nabla e_I=0$ for any $I$, we obtain $d^\nabla\eta=0=d^v0$, so $d_2[[\eta]_{d^v}]_{d^\nabla}=[[-d^\nabla 0]_{d^v}]_{d^\nabla}=0$. Moreover,  by \eqref{diffr} and by the same arguments used to prove \ref{teorema}, each $\xi_1, \ldots, \xi_{r-1}$ can step by step be chosen to be 0 and thus $d_r \equiv 0$, for $r \geq 2$. We proved thus that $E_2=E_\infty$. 
\end{proof}

We proceed now with the proof of \ref{teoremaB}:
\begin{proof}
Since $E_r^{\bullet, \bullet}$ converges to $H^\bullet_\theta(X,\CC)$ and $E_2=E_\infty$, then 
\begin{equation*}
H^l_\theta(X,\CC) \cong \oplus_{p+q=l} E_2^{p, q} \cong\oplus_{\substack{p+q=l\\I\in\mathcal{I}_{q}\\ \rho^\theta\otimes\sigma_I\equiv 1}} H^p(\TT^s,\CC) \otimes e_I. 
\end{equation*}
Finally, in order to represent $H^l_\theta(X,\CC)$ by $U$ invariant forms on $\hat X$, we need to tensorize with a global frame  $\xi$ of $L_\theta$.  If $\theta=\sum_{k=1}^sa_kd\ln v_k$, then $\xi$ is given by  $\xi=\prod_{k=1}^sv_k^{a_k}$ on $\hat X$, and so the conclusion follows.

\end{proof}

\begin{remark} We want to draw attention to the fact that for both spectral sequences involved in our proofs, the isomorphism $E_2^{\bullet,\bullet}\cong H^\bullet(B)\otimes H^\bullet(F)$ alone was not enough to imply 
the degeneracy of $E_r^{\bullet, \bullet}$ at page $E_2$. An example of fiber bundle $F \rightarrow X \rightarrow B$  for which this isomorphism at the second page holds, but whose corresponding Leray-Serre spectral sequence does not degenerate at $E_2$ is 
given by the Hopf fibration $S^1 \rightarrow S^{2n+1} \rightarrow \C\mathbb{P}^n$.
\end{remark}

\section{Applications and Examples} 

Let us start this section by giving the immediate consequence of \ref{teoremaA}, which is the explicit cohomology of OT manifolds when there are no trivial representations other than the obvious ones:
\begin{corollary}\label{trivialI}
Let $(K,U)$ be a number field together with an admissible group of units $U\subset K$ so that $U$ admits no trivial representations $\sigma_I$ other than the ones corresponding to $I=\emptyset$ and $I=(1,2,\ldots, n)$, and let $X$ be the OT manifold associated to $(K,U)$. The Betti numbers of $X$ are:
\begin{align*}
b_l=b_{2m-l}&=\binom sl \text{ for } 0\leq l\leq s\\
b_l&=0 \text{ for } s<l<n.
\end{align*}
\end{corollary}

\begin{corollary}
For an OT manifold of type $(s,t)$, all Betti numbers $b_l$  for $0\leq l\leq s$ and for $2m-s\leq l\leq 2m$ are positive. 
\end{corollary}
\begin{proof}
For $0\leq l\leq s$, $H^l(\TT^s,\CC)$ is a summand of $H^l(X,\CC)$, corresponding to $p=l$ and $I=\emptyset$, so $\sigma_I\equiv 1$. Hence:
\begin{equation*}
b_l(X)\geq b_l(\TT^s)=\binom sl >0.
\end{equation*}
The assertion follows for $2m-s\leq l\leq 2m$ by the Poincar\'e duality. 
\end{proof}

We have computed the cohomology algebras of an OT manifold $X(K,U)$ in terms of numerical invariants associated to $(K,U)$, namely in terms of the trivial representations $\sigma_I$ of $U$. Clearly, if $(K,U)$ is not simple, in the sense that there exists an intermediate field extension $\Q\subset K'\subset K$ so that $U\subset K'$, then there exist trivial representations $\sigma_I:U\rightarrow \CC^*$ with $0<|I|=[K':\Q]<[K:\Q]$. It would be interesting to know whether the converse is true, i.e. if $(K,U)$ is of simple type, is the set $\{I|\sigma_I:U\rightarrow \CC^*, \sigma_I\equiv1\}$ only formed by $\emptyset$ and $I=(1,\ldots, n)$? Let us note that in \cite[Proposition 2.3]{ot}, the second Betti number of an OT manifold of simple type was computed, and coincides with ours when there are no other trivial representations, implying an affirmative answer for the above question when $|I|=2$. We do not address this problem in the present article, but we give an example where the answer is affirmative,  allowing us to give the explicit Betti numbers of the corresponding manifold:

\begin{example}\label{exemplu} Let $p$ be any odd prime number and take the polynomial $f=X^p-2\in\Q[X]$. This polynomial has one real root $\sqrt[p]{2}$ and the complex roots $\sqrt[p]{2}\epsilon, \ldots, \sqrt[p]{2}\epsilon^{p-1}$, where $\epsilon$ is a $p$-th root of unity. Let $K=\Q(\sqrt[p]{2})$, which is of type $(1,\frac{p-1}{2})$. We notice first that $u=\sqrt[p]{2}-1$ is a unit of $O_K$ since its norm, which is the product of all the embeddings of $u$ in $\CC$,  is equal to $1$:
\begin{equation*}
(\sqrt[p]{2}-1) \ldots (\sqrt[p]{2}\epsilon^{p-1}-1) = (-1)^pf(1)=1.
\end{equation*} 
Since $u$ is also clearly positive, we can then take $U$ to be generated by $u$. Let then $X=X(K,U)$ be the corresponding OT manifold. We claim that there is no index $I$ with $p > |I| \geq 2$ and $\sigma_I \equiv 1$. By \ref{trivialI}, this will imply that the Betti numbers of $X$ will verify $b_0=b_{p+1}=b_1=b_{p}=1$ and $b_i=0$ for any $i \neq 0, 1, p, p+1$. 

 Let us assume by contradiction the existence of such $I = (1 \leq i_1 < \ldots < i_k \leq p)$, with $k<p$. For any $1 \leq j \leq p$, we denote by $\sigma_j$ the embedding of $K$ into $\CC$ mapping $\sqrt[p]{2}$ to $\sqrt[p]{2}\epsilon^{j-1}$. Then $\sigma_I \equiv 1$ rewrites as:
\begin{equation*}
(\sqrt[p]{2}\epsilon^{i_1-1}-1)(\sqrt[p]{2}\epsilon^{i_2-1}-1)\ldots(\sqrt[p]{2}\epsilon^{i_k-1}-1)=1,
\end{equation*}
equivalent to:
\begin{equation}\label{poly}
a_0\sqrt[p]{2^k}- a_1\sqrt[p]{2^{k-1}} + \ldots + (-1)^{k-1}a_{k-1}\sqrt[p]{2} + (-1)^k-1=0
\end{equation}
where $a_l=\sum_{j_1 < \ldots < j_l} \epsilon^{i_1 + \ldots + \widehat{i_{j_1}}+ \ldots + \widehat{i_{j_l}}+ \ldots i_k-k+l}$, and the symbol $\hat{\cdot}$ over an element marks the fact that the element is missing. Let $g$ be the polynomial: 
\begin{equation*}
g=a_0X^k - a_1X^{k-1}+\ldots+(-1)^{k-1}a_{k-1}X+(-1)^k-1 \in \Q(\epsilon)[X].
\end{equation*} 
Then \eqref{poly} implies $g(\sqrt[p]{2})=0$, hence $g$ is a multiple of the minimal polynomial of $\sqrt[p]{2}$ over the field $\Q(\epsilon)$. We prove next that this polynomial is actually $X^p-2$. Indeed, we have the following two intermediate extensions:
\begin{equation*}
\Q \subset \Q(\epsilon) \subset \Q(\epsilon, \sqrt[p]{2})
\end{equation*}
\begin{equation*}
\Q \subset \Q(\sqrt[p]{2}) \subset \Q(\epsilon, \sqrt[p]{2}).
\end{equation*}
We thus get 
\begin{equation}\label{produs}
[\Q(\epsilon, \sqrt[p]{2}) : \Q] = [\Q(\epsilon, \sqrt[p]{2}) : \Q(\epsilon)] \cdot [\Q(\epsilon) : \Q] = [\Q(\epsilon, \sqrt[p]{2}) :\Q(\sqrt[p]{2})] \cdot [\Q(\sqrt[p]{2}) : \Q].
\end{equation}
Since $X^p-2 \in \Q(\epsilon)[X]$, we have $[\Q(\epsilon, \sqrt[p]{2}) : \Q(\epsilon)] \leq p$. 
In general, if $\epsilon$ is an $n$-th root of unity, $[\Q(\epsilon) : \Q]=\phi(n)$, where $\phi(n)$ is Euler's function. In our case $\phi(p)=p-1$, whence \eqref{produs} implies:
\begin{equation*}  
[\Q(\epsilon, \sqrt[p]{2}) : \Q] =(p-1)[\Q(\epsilon, \sqrt[p]{2}) : \Q(\epsilon)]=p\cdot[\Q(\epsilon, \sqrt[p]{2}) :\Q(\sqrt[p]{2})].
\end{equation*}
As $p$ and $p-1$ are relatively prime, we get moreover that $p$ divides $ [\Q(\epsilon, \sqrt[p]{2}) : \Q(\epsilon)]$, therefore $p=[\Q(\epsilon, \sqrt[p]{2}) : \Q(\epsilon)]$.  Thus the minimal polynomial of $\sqrt[p]{2}$ over the field $\Q(\epsilon)$ is $X^p-2$, contradicting the fact that $k=\deg g< p$.
\end{example}

\hfill

We can also obtain, via \ref{trivialI}, the explicit de Rham cohomology algebra of OT manifolds of type $(s,1)$. These manifolds are known to admit \textit{locally conformally \Ka} (LCK) metrics, which by definition are Hermitian metrics induced by \Ka\ metrics on the universal cover on which the fundamental group acts by homotheties. Equivalently, they are Hermitian metrics whose fundamental form $\Omega$ verifies $d\Omega=\theta\wedge\Omega$, where $\theta$ is a closed real one-form on the manifold, called \textit{the Lee form}. Such metrics were constructed in \cite{ot} on all OT manifolds of type $(s,1)$, but in general it is still an open problem to decide whether LCK metrics might exist on other types of OT manifolds. The existence of an LCK metric translates into a condition on the numerical data $(K,U)$: if $X(K,U)$ admits an LCK metric, then, by \cite[Proposition 2.9]{ot}, for any $u\in U$ we have:
\begin{equation}\label{LCK}
r(u)^2:=|\sigma_{s+1}(u)|^2=\ldots=|\sigma_{s+t}(u)|^2=(\sigma_{1}(u)\cdots\sigma_{s}(u))^{-1/t}.
\end{equation}
In the appendix to \cite{d} of L. Battisti, it was shown in Theorem 8, p.271, that \eqref{LCK} is also a sufficient condition for an LCK metric to exist. 

\begin{proposition}\label{LCKbun}
Let $X$ be an OT manifold of type $(s,t)$ admitting some LCK metric. Its de Rham cohomology algebra $H^\bullet(X,\CC)$ is isomorphic to the graded algebra over $\CC$ generated by: 
\begin{equation*}
d\ln v_1,\ldots, d\ln v_s, dz_1\wedge d\cz_1\wedge\ldots dz_t\wedge d\cz_t\wedge dr_1\wedge\ldots\wedge dr_s.
\end{equation*}
In particular, its Betti numbers are:
\begin{align*}
b_l=b_{2m-l}&=\binom sl \text{ for } 0\leq l\leq s\\
b_l&=0 \text{ for } s<l<n.
\end{align*}
\end{proposition}
\begin{proof}
By \ref{trivialI}, it suffices to show that $U$ admits no trivial representations $\sigma_I$ other than the two obvious ones. So let $I=(0<i_1<\ldots<i_k\leq n)$ with $k>0$ and $\sigma_I \equiv 1$. After eventually renumbering the coordinates, we can suppose without loss of generality that $I$ is of the form 
\begin{equation*}
I=(1,\ldots q, j_1,\ldots j_p, s+t+1,\ldots s+t+l),
\end{equation*} 
with $0\leq q\leq s<j_1<\ldots<j_p\leq s+t$ and $0\leq p,l\leq t$.

Since $\sigma_I=1$ we have $|\sigma_I|=1$ which, together with \eqref{LCK}, gives the relation:
\begin{equation*}
(\sigma_{1}\cdots\sigma_{q})^{-1}=r^{l+p}=(\sigma_1\cdots\sigma_s)^{-\frac{l+p}{2t}}.
\end{equation*}
As $\sigma_1,\ldots, \sigma_s$ are $\RR$-linearly independent, this relation must be the trivial one, implying that $l+p=2t$ and $q=s$, so $I=(1,\ldots, n)$, which finishes the proof. 
\end{proof}

In LCK geometry, it is interesting to know also the  twisted  cohomology with respect to the Lee form of the LCK metric. For the OT manifolds, we first determine the set of all possible de Rham classes of Lee forms of LCK metrics, then compute the corresponding  twisted  cohomology. The result that follows generalizes the result in \cite{alex}, where it is proven that the set of possible Lee classes of LCK metrics on Inoue surfaces of type $\mathcal{S}^0$, namely OT-manifolds of type $(1, 1)$, has only one element.

\begin{proposition}\label{unic}
Let $X=X(K,U)$ be an OT manifold of type $(s,t)$. There exists at most one Lee class of an LCK metric on $X$, namely the one represented by the $U\ltimes O_K$-invariant form on $\HH^s\times\CC^t$, $\theta=\frac{1}{t}d\ln (\prod_{k=1}^sv_k)$.
\end{proposition}
\begin{proof}
First note that, as $H^1(X,\RR)\cong\Hom(\pi_1(X),\RR)$, we can identify a de Rham class $[\eta]_{dR}$ with a group morphism $\tau:\pi_1(X)\rightarrow \RR$. The corresponding morphism $\tau$ is precisely the automorphy representation: if $\eta=d\phi$ on the universal cover $\tilde X$, then $\tau$ is given by $\tau(\gamma)=\gamma^*\phi-\phi$, for any $\gamma\in\pi_1(X)$.

Moreover, if $X$ admits some LCK metric $(\Omega,\eta)$ with $\eta=d\phi$ on $\tilde X$, and if $\Omega_K:=\e^{-\phi}\Omega$ is the corresponding \Ka\ form on $\tilde X$, then $\tau=\tau_{[\eta]}$ is also determined by: $\gamma^*\Omega_K=\e^{-\tau(\gamma)}\Omega_K$ for any $\gamma\in\pi_1(X)$. Hence, it suffices to show that for any \Ka\ metric $\Omega_K$ on $\tilde X$ inducing an LCK metric on $X$, the automorphy representation determined by $\Omega_K$ is precisely the one corresponding to $\theta$, namely:
\begin{align*}\label{rho}
\tau_\theta(a)&=0  \text{ for } a\in O_K \\
\tau_\theta(u)&=\frac{1}{t}\sum_{k=1}^s \ln \sigma_k(u)  \text { for } u\in U.
\end{align*}

Let now $\Omega_K$ be a \Ka\ metric on $\tilde X$ on which $\pi_1(X)$ acts by homotheties, and denote by $\tau$ the corresponding representation described before. We recall that under the abelianization morphism $U\ltimes O_K\rightarrow H_1(X,\ZZ)$, $O_K$ maps to a finite group. This implies that $O_K$ will act by isometries on $\Omega_K$. Hence, the form $\Omega_K$ descends to the manifold $\hat X:=\HH^s\times \CC^t/O_K$. But the torus $\TT:=\RR^{2t+s}/O_K$ acts holomorphically by translations on $\hat X$, so we can average $\Omega_K$ over $\TT$ to get a new $\TT$-invariant \Ka\ form on $\hat X$:
\begin{equation*}
\Omega'_K:=\int_{\TT}a^*\Omega_K\mu(a)
\end{equation*}
where $\mu$ is the constant volume form on $\TT$ with $\int_\TT\mu=1$. The automorphy of $\Omega'_K$ is also $\tau$, as for any $u\in U$ we have:
\begin{equation*}
u^*\Omega_K'=\int_\TT(au)^*\Omega_K\mu(a)=\int_\TT(uc_u(a))^*\Omega_K\mu(a)=\int_\TT c^*_u(a)(\e^{-\tau(u)}\Omega_K)\mu(c_u(a))=\e^{-\tau(u)}\Omega_K'.
\end{equation*}

Now  write $\Omega_K'=\Omega_0+\Omega_{01}+\Omega_1$ with respect to the splitting 
\begin{equation}\label{splitt}
\textstyle\bigwedge^2_{\tilde X}=\textstyle\bigwedge^{2}_{\CC^t}\oplus (\textstyle\bigwedge^1_{\CC^t}\otimes \textstyle\bigwedge^1_{\HH^s})\oplus\textstyle\bigwedge^2_{\HH^s}
\end{equation}
and also split $d=d_0+d_1$, with $d_0$ being the differentiation with respect to the $\CC^t$-variables and $d_1$, the $\HH^s$-variables. The $\CC^t$-invariance of $\Omega'_K$ implies that $d_0\Omega_K'=0$. The condition $d\Omega_K'=0$ then gives, on the $\bigwedge^2_{\CC^t}\otimes \bigwedge^1_{\HH^s}$-component, $d_1\Omega_0=0$. So $\Omega_0=\sum_{ij} f_{ij}dz_i\wedge d\cz_j$, with $f_{ij}\in\CC$ for any $1\leq i,j \leq t$. 

Now, if $u\in U$, $u^*\Omega'_K=\e^{-\tau(u)}\Omega'_K$ implies that:
\begin{equation*}
f_{ij}\sigma_{s+i}(u)\ov{\sigma_{s+j}(u)}=f_{ij}\e^{-\tau(u)} \text{ for any } 1 \leq i,j \leq t.
\end{equation*} 

In particular, since $f_{ii}\neq 0$ for any $1 \leq i \leq t$, we have: $\tau(u)=-\ln|\sigma_{s+1}(u)|^2=\ldots=-\ln|\sigma_{s+t}(u)|^2$. But we also have $\prod_{k=1}^{s}\sigma_k\prod_{j=s+1}^{s+t}|\sigma_j|^2=1$. This implies that $\tau=\tau_\theta$, and the conclusion follows.
\end{proof}

This last result has an immediate corollary concerning the stability of LCK metrics on OT manifolds, as studied by R. Goto in \cite{g}. In order to state it, let us first note that, given an LCK structure $(\Omega,\theta)$ on $X$, we have an associated flat line bundle $(L, \nabla)=L_{[\theta]}$, and $\Omega$ can be seen as a $d^\nabla$-closed section of $\Omega^2_X\otimes L^*$. Conversely, given a flat line bundle $(L,\nabla)$ over $X$, an $L$-valued positive $(1,1)$-form which is $d^\nabla$-closed induces an LCK structure on $X$.

\begin{corollary}\label{stab}
On an OT manifold of LCK type, the LCK structure is not stable under small deformations of flat line bundles. More specifically, if $(\Omega, L, \nabla)$ is an LCK structure on an OT manifold $X$, $\epsilon>0$ and $\{L_v\}$ is a non-trivial analytic deformation of flat line bundles for $|v|<|\epsilon|$ with $L_0=L$, then for any $0<|v|<\epsilon$,  there are no $L_v$-valued LCK structures.
\end{corollary}

Next we compute the explicit  twisted  cohomology groups with respect to the Lee form:

\begin{proposition}\label{twistedLCK} Let $X$  be an OT manifold of type $(s, t)$ admitting an LCK metric and let $\theta = \frac{1}{t}\sum_{k=1}^s d\ln v_k$. Then for any $0\leq l\leq 2m$, we have: 
\begin{equation*}
H^l_\theta(X)\cong(v_1 \cdots v_s)^{\frac{1}{t}}\oplus_{j=1}^t\CC dz_j\wedge d\cz_j\otimes\textstyle\bigwedge^{l-2}\CC\{d\ln v_1,\ldots,d\ln v_s\}.
\end{equation*}
In particular, the corresponding twisted Betti numbers are given by $\dim_\CC H^l_\theta(X)=t\binom s{l-2}$ for any $0\leq l \leq 2m$. 

\end{proposition}

\begin{proof}
In order to apply \ref{teoremaB}, we need to identify, for any $0\leq k\leq n$, the set corresponding to $[\theta]_{dR}$:
\begin{equation*}
\mathcal{J}_k=\{I=(1 \leq i_1< i_2 \ldots < i_k\leq n) \mid \sigma_1^{\frac{1}{t}}\ldots \sigma^{\frac{1}{t}}_s\sigma_{i_1}\ldots\sigma_{i_k}=1\}.
\end{equation*}
We shall prove that $\mathcal J_k=\emptyset$ for $k\neq 2$ and $\mathcal J_2=\{(s+j,s+t+j)| 1 \leq j \leq t\}$.

Let us fix $k$ and $I\in\mathcal J_k$. As before, we can assume that $I$ is of the form:
\begin{equation*}
I=(1, \ldots, q, j_1, \ldots, j_p, s+t+1, \ldots, s+t+l)
\end{equation*}  
such that $0\leq q \leq s<j_1<\ldots<j_p\leq s+t$, $0\leq p,l \leq t$ and $q+p+l=k$. Then $|(\sigma_1\cdots\sigma_s)^{1/t}\sigma_I|=1$ together with \eqref{LCK} implies:
\begin{equation*}
\sigma_1^{\frac{1}{t}+1}\ldots\sigma_q^{\frac{1}{t}+1}\sigma_{q+1}^{\frac{1}{t}} \ldots \sigma_s^{\frac{1}{t}}=r^{-(p+l)}=(\sigma_1\cdots\sigma_s)^{\frac{p+l}{2t}}.
\end{equation*}

By the $\R$ linear independence of $\sigma_1, \ldots, \sigma_s$, this must be the trivial relation. If $0<q<s$, we would get that $\frac{1}{t}+1=\frac{p+l}{2t}=\frac{1}{t}$, which is a contradiction. If $q=s$, then we would get that $\frac{1}{t}+1=\frac{p+l}{2t}$, or also $p+l=2t+2$, contradicting the fact that $p+l\leq 2t$. Hence $q=0$ and we get the relation $\frac{1}{t}=\frac{p+l}{2t}$, or also $p+l=2=k$. In particular, $\mathcal J_k=\emptyset$ for $k\neq 2$.

Let us note now that the set $\{(s+j,s+t+j)|j=1,t\}$ is included in $\mathcal J_2$. In order to show that these are all the possible multi-indexes, let $I=(i_1<i_2)\in\mathcal J_2$. We already showed that $i_1 > s$. Since $\sigma_{i_1}\sigma_{i_2}=(\sigma_1\cdots\sigma_s)^{-1/t}$ is real, we get that $\sigma_{i_1}\sigma_{i_2}=\overline{\sigma_{i_1}}\overline{\sigma_{i_2}}$. Combining with $|\sigma_{i_1}|=|\sigma_{i_2}|$, we obtain $\sigma_{i_1}^2=\overline\sigma_{i_2}^2$, therefore $\sigma_{i_1} = \pm \overline{\sigma}_{i_2}$. The case $\sigma_{i_1}= -  \overline{\sigma}_{i_2}$ is excluded, because this would give the following contradiction:
\begin{equation*}
0 > -  \overline{\sigma}_{i_2} \sigma_{i_2} = (\sigma_1\ldots\sigma_s)^{-\frac{1}{t}} > 0. 
\end{equation*}

So $\sigma_{i_1}=\overline{\sigma}_{i_2}$. But there exists $s+t\leq j\leq s+2t$ with $|i_1-j|=t$ and $\sigma_{i_1}=\ov{\sigma}_j$, so $\sigma_{i_2}=\sigma_j$. We want to show that $i_2=j$, i.e. $I=(j-t,j)$. 

Consider $M$ the $\ZZ$-submodule of $O_K$ generated by $U$, which is a subring of $O_K$, and let $K'$ be its fraction field. We have $U\subset M\subset K'\subset K$, and we showed in the above proposition that $U$ has no trivial representations, so in particular $(K,U)$ is simple, thus $K'=K$. But the relation $\sigma_{i_2}=\sigma_j$ extends to $M$, and so also to $K'=K$. This last fact is possible only if $i_2=j$.
\end{proof}

\begin{remark} Notice that since $H^l_\theta(X)$ does not vanish and $\theta$ is real-valued, by the result of \cite{llmp} $\theta$ is not parallel with respect to any  metric $g$ on $X$.
\end{remark}

\begin{remark} In \cite{k}, OT manifolds are given a solvmanifold structure, namely they are shown to be of the form $\Gamma \setminus G$, where $G$ is a solvable Lie group and $\Gamma$ is a co-compact lattice in $G$. Consequently, one can consider the cohomologies $H^\bullet(\mathfrak{g})$ and $H^\bullet_\theta(\mathfrak{g})$, where $\mathfrak{g}$ is the Lie algebra of $G$ and $\theta$ is a closed $G$-invariant form. A natural question is then: does one have isomorphisms $H^\bullet_{dR}(X(K, U)) \cong H^\bullet(\mathfrak{g})$ and $H^\bullet_{\theta}(X(K, U)) \cong H_{\theta}^\bullet(\mathfrak{g})$? For a general solvmanifold, this does not always hold. However, H. Kasuya proved in \cite[Example 4]{k2} that on OT manifolds of type $(s,1)$, this isomorphism is valid for the de Rham cohomology. In \cite[Theorem 4.3]{aot}, it is proved that in the twisted cohomology,  the isomorphism holds for a subclass of $X(K, U)$ of type $(s, 1)$, satisfying the so-called {\em Mostow condition}. Finally, since in \ref{teoremaB} we represented the corresponding cohomologies by invariant forms with respect to the action of $G$ described in \cite{k}, we obtain as a consequence that for all OT manifolds $X$ of type $(s, t)$, we have the isomorphism $H^k_\theta(X) \cong H^k_\theta(\mathfrak{g})$, although they might not all satisfy the Mostow condition.
\end{remark}

In \cite{alex} it was proven that there are no $d_\theta$-exact metrics on OT manifolds of type $(s, 1)$. We give next a generalization of this result, in which we determine all the possible LCK classes in $H^2_\theta$. As a consequence of this, we obtain a hard Lefschetz-type theorem associated to an LCK metric on an OT manifold.

\begin{corollary}\label{LCKclass}
Let $X$ be an OT manifold of type $(s,t)$ with an LCK structure $(\Omega,\theta)$, where $\theta= \frac{1}{t}\sum_{k=1}^sd\ln v_k$. Then the  twisted  class of $\Omega$ in $H^2_\theta(X)$ is of the form: 
\begin{equation*}
(v_1\cdots v_s)^{\frac{1}{t}}\sum_{j=1}^ta_jidz_j\wedge d\ov z_j, \ \ \ a_j\in\RR_{>0} \ \forall j\in\{1,\ldots, t\}.
\end{equation*} 
In particular, if we let $\mathrm{Lef}_{\Omega}$ denote the Lefschetz operator $\mathrm{Lef}_\Omega=\Omega\wedge \cdot$, then for any $0\leq l\leq 2m-2$, $\mathrm{Lef}_\Omega$ induces a morphism in cohomology: 
\begin{equation*}
[\mathrm{Lef}_\Omega]:H^l(X,\CC)\rightarrow H^{l+2}_{\theta}(X)
\end{equation*}
which is injective for $0\leq l\leq m$ and surjective for $m\leq l\leq 2m-2$.
\end{corollary}
\begin{proof}
Let us start by noting that, as in the case of the de Rham cohomology, the twisted cohomology with respect to $\theta$ is the twisted cohomology of $\TT$-invariant forms. This is a direct consequence of \ref{teoremaB}, but can also be seen by an argument completely analogous to \ref{Gpoincare} and using the fact that $\theta$ vanishes on vector fields tangent to $\TT^n$. Hence, by averaging the form $\Omega$ to a $\TT$-invariant LCK form $\Omega'$ as in \ref{unic}, the  twisted  class does not change: $[\Omega]_\theta=[\Omega']_\theta\in H^2_\theta(X)$. 

At the same time, we saw that the corresponding \Ka\ form $\Omega'_K$ writes with respect to the splitting \eqref{splitt} as  $\Omega'_K=\Omega_0+\Omega_{01}+\Omega_1$, with $\Omega_0$ a constant positive form on $\CC^t$. Also, given the expression of $\theta$, we have $\Omega'=(v_1\cdots v_s)^{1/t}\Omega_K':=\omega_0+\omega_{01}+\omega_1$, where again $\Omega'$ was decomposed with respect to the splitting \eqref{splitt}. Clearly, $d_\theta\omega_0=0$, so also $d_\theta(\omega_{01}+\omega_1)=0$, thus we can write $[\Omega']_\theta=[\omega_0]_\theta+[\omega_{01}+\omega_1]_\theta\in H^2_\theta(X)$. Now, since by \ref{twistedLCK}, we have: 
\begin{equation}\label{pp}
H^2_\theta(X)\cong (v_1\cdots v_s)^{\frac{1}{t}}\oplus_{j=1}^t\CC dz_j\wedge\ d\cz_j,
\end{equation}
it follows that $[\omega_{01}+\omega_1]=0\in H^2_\theta(X)$. Indeed, otherwise we would have that  on $\tilde X$, $\omega_{01}+\omega_1+d_{\theta}\eta$ is valued in $\bigwedge^2_{\CC^t}$ for some one-form $\eta\in \Omega^1_X(X)$, which is impossible. Hence $[\Omega]_\theta=[\omega_0]_\theta=\omega_0$ under the isomorphism \eqref{pp}, so the first assertion follows. The second assertion follows from the description of the cohomology groups given in \ref{LCKbun} and \ref{twistedLCK} and from the non-degeneracy of $[\Omega]$.
\end{proof}

\begin{remark}
The fact that for any LCK form $\Omega$ on $X$, the operator $\mathrm{Lef}_\Omega:H^1(X,\CC)\rightarrow H^3_\theta(X)$ is injective also implies \ref{stab} via \cite[Theorem 2.4]{g}.
\end{remark}

We end this section with one more application concerning the possible real Chern classes of vector bundles on OT manifolds:
\begin{proposition}\label{Chern}
Let $X(K,U)$ be an OT manifold of type $(s,t)$  verifying that $U$ admits no trivial representations $\sigma_I$ unless $|I|\in\{0, n\}$. Then, for any $1\leq k<n/2$, every $d$-closed real $(k,k)$ form on $X$ is exact. In particular, if $E$ is a complex vector bundle on $X$, its first $[(n-1)/2]$ real Chern classes $c_k(E)^\RR\in H^{2k}(X,\RR)$ vanish.
\end{proposition}
\begin{proof}
By \ref{trivialI}, we deduce that:  
\begin{align*}
H^{2k}(X,\RR)&\cong\textstyle\bigwedge^{2k}\RR\{f_1,\ldots, f_s\} \  \ \ \ \ \ \text{ for }2k<n
\end{align*}
where $f_l:=v_l^{-1}dv_l$ for $1 \leq l \leq s$. 
Let us also denote by $\phi_l=-\frac{i}{2}v_l^{-1}d w_l=f_l^{1,0}$ for $1 \leq l \leq s$, so that $f_l=\phi_l+\ov{\phi}_l$.

Let $\al$ be a real closed $(k,k)$ form on $X$. By the above, we can write: $\al=\sum_{I\in\mathcal{I}_{2k}}a_If_I+d\be$, where for every multi-index $I=(i_1<\ldots<i_{2k})$, $f_I=f_{i_1}\wedge\ldots\wedge f_{i_{2k}}$, $a_I\in\RR$ and $\be\in\Omega^{2k-1}_X(X)$ is a real form. In particular, in bidegree $(2k,0)$, this reads: 
\begin{equation*}
\al^{2k,0}=0=\sum_{I\in\mathcal I_{2k}}a_I\phi_I+\del\be^{2k-1,0}.
\end{equation*} 
But, for any $I$, $\phi_I$ is not $\del$-exact, and neither is the sum $\sum_I a_I\phi_I$, unless it is zero. In order to see this, one could for instance choose a hermitian metric on $X$ defining an $L^2$ adjoint operator $\del^\star$ with respect to which one would have $\del^\star\phi_I=0$ for any $I$. It would follow then that each $\al_I$ is $L^2$-orthogonal to $\Im\del$, and so $\sum_I a_I\phi_I=\del\be^{2k-1,0}=0$. In particular, this implies that $a_I=0$ for each $I\in\mathcal{I}_{2k}$, and so $\al=d\be$.
\end{proof}
\begin{remark}
In the literature specialized on topology, there is a complex called Morse-Novikov, associated to a closed one-form $\theta$ of Morse type, i.e. locally given by the differential of a Morse function. It was first considered by Novikov in \cite{n1} and \cite{n2}, and for a thorough description we refer to \cite{farber}. The construction of this complex is based on the number of zeros of $\theta$, just as the Morse-Smale complex of a Morse function $f$ is based on the number of zeros of $f$ and actually these two complexes coincide when $\theta=df$. If $\theta$ is a nowhere vanishing one-form, as the Lee form   in \ref{twistedLCK} is, the Morse-Novikov complex is trivial, therefore its cohomology vanishes. However, the twisted cohomology does not vanish, as our computation indicates; consequently, OT manifolds provide examples in all dimensions of spaces for which these two cohomologies differ.
\end{remark}

\noindent{\bf Acknowledgements:} We are very grateful to Victor Vuletescu for many insightful discussions and to Andrei Moroianu, Mihaela Pilca and Massimiliano Pontecorvo for a careful reading of a first draft of the paper and useful suggestions and remarks.

\end{document}